\documentclass{article}
\usepackage{graphicx} 

\usepackage{authblk}

\usepackage{float}

\usepackage[dvipsnames]{xcolor}
\usepackage{xcolor,todonotes}
\usepackage{enumitem}

\usepackage{tikz}
\tikzstyle{edge}=[shorten <=2pt, shorten >=2pt, >=stealth, line width=1.5pt]
\tikzstyle{vertex}=[circle, fill=white, draw, minimum size=6pt, inner sep=0pt,
                    outer sep=0pt]

\usepackage{amsmath}
\usepackage{amssymb,hyperref}
\usepackage{bm} 

\usepackage{amsthm,geometry,cite}

\newtheorem{theorem}{Theorem}
\newtheorem{definition}{Definition}
\newtheorem{conjecture}{Conjecture}
\newtheorem{lemma}[theorem]{Lemma}
\newtheorem{corollary}[theorem]{Corollary}

\newtheorem{observation}[theorem]{Observation}
\newtheorem{example}[theorem]{Example}
\newtheorem{question}[theorem]{Question}
\newtheorem{problem}[theorem]{Problem}

\DeclareMathOperator{\Th}{Th}

\DeclareMathOperator{\Sep}{Sep}
\DeclareMathOperator{\Bet}{Betw}
\DeclareMathOperator{\Age}{Age}

\newcommand{\bA}{\mathbb A}
\newcommand{\bB}{\mathbb B}
\DeclareMathOperator{\Csp}{CSP}
\DeclareMathOperator{\Pcsp}{PCSP}
\DeclareMathOperator{\UMI}{\textbf{UM1}}
\DeclareMathOperator{\UMII}{\textbf{UM2}}
\DeclareMathOperator{\UMIII}{\textbf{UM3}}
\DeclareMathOperator{\UMIV}{\textbf{UM4}}
\DeclareMathOperator{\UMV}{\textbf{UM5}}
\DeclareMathOperator{\UMVI}{\textbf{UM6}}
\DeclareMathOperator{\UMVII}{\textbf{UM7}}
\DeclareMathOperator{\UMVIII}{\textbf{UM8}}

\DeclareMathOperator{\UEI}{\textbf{UE1}}
\DeclareMathOperator{\UEII}{\textbf{UE2}}

\DeclareMathOperator{\UBI}{\textbf{UB1}}
\DeclareMathOperator{\UBII}{\textbf{UB2}}
\DeclareMathOperator{\UBIII}{\textbf{UB3}}
\DeclareMathOperator{\UBIV}{\textbf{UB4}}
\DeclareMathOperator{\USI}{\textbf{US1}}
\DeclareMathOperator{\USII}{\textbf{US2}}
\DeclareMathOperator{\USIII}{\textbf{US3}}
\DeclareMathOperator{\USIV}{\textbf{US4}}

\DeclareMathOperator{\EX}{\textbf{EX}}
\DeclareMathOperator{\EXIa}{\textbf{EX1a}}
\DeclareMathOperator{\EXIb}{\textbf{EX1b}}
\DeclareMathOperator{\EXIc}{\textbf{EX1c}}

\DeclareMathOperator{\EXIIa}{\textbf{EX2a}}
\DeclareMathOperator{\EXIIb}{\textbf{EX2b}}
\DeclareMathOperator{\EXIIc}{\textbf{EX2c}}

\DeclareMathOperator{\EXIIIa}{\textbf{EX3a}}
\DeclareMathOperator{\EXIIIb}{\textbf{EX3b}}
\DeclareMathOperator{\EXIIIc}{\textbf{EX3c}}
\DeclareMathOperator{\EXIIId}{\textbf{EX3d}}

\DeclareMathOperator{\EXIVa}{\textbf{EX4a}}
\DeclareMathOperator{\EXIVb}{\textbf{EX4b}}
\DeclareMathOperator{\EXIVc}{\textbf{EX4c}}
\DeclareMathOperator{\EXIVd}{\textbf{EX4d}}
\DeclareMathOperator{\EXIVe}{\textbf{EX4e}}
\DeclareMathOperator{\EXIVf}{\textbf{EX4f}}

\title{The Generic Circular Triangle-Free Graph}

\author[1]{Manuel Bodirsky\thanks{manuel.bodirsky@tu-dresden.de}}
\author[1]{Santiago Guzm\'an-Pro\thanks{santiago.guzman\_pro@tu-dresden.de\\
Both authors have been funded by the European Research Council (Project POCOCOP, ERC Synergy Grant 101071674). Views and opinions expressed are however those of the authors only and do not necessarily reflect those of the European Union or the European Research Council Executive Agency. Neither the European Union nor the granting authority can be held responsible for them.}}

\affil[1]{Institut f\"ur Algebra, TU Dresden}
\date{\today}

\begin{document}

\maketitle

\begin{abstract}
  
    In this paper, we introduce the generic circular triangle-free graph $\mathbb C_3$
    and propose a finite axiomatization of its first order theory. In particular, our main
    results show that a countable graph $G$ embeds into $\mathbb C_3$ if and only if it is a
    $\{K_3, K_1 + 2K_2, K_1+C_5, C_6\}$-free graph. As a byproduct of this result, we obtain
    a geometric characterization of finite $\{K_3, K_1 + 2K_2, K_1+C_5, C_6\}$-free graphs, and
    the (finite) list of minimal obstructions of unit Helly circular-arc graphs with independence
    number strictly less than three.
    
    The circular chromatic number $\chi_c(G)$ is a refinement of the classical
    chromatic number $\chi(G)$. 
    We construct $\mathbb C_3$ so that a graph $G$ has circular chromatic number strictly less
    than three if and only if $G$ maps homomorphically to $\mathbb C_3$.
    We build on our main results to show that $\chi_c(G) < 3$ if and only if $G$ can be extended to
    a $\{K_3, K_1 + 2K_2, K_1+C_5, C_6\}$-free graph, and in turn, we use this result to reprove an
    old characterization of $\chi_c(G) < 3$ due to Brandt (1999). Finally, we answer a question recently asked by
    Guzm\'an-Pro, Hell, and Hern\'andez-Cruz by showing that the problem of deciding for a given
    finite graph $G$ whether $\chi_c(G) < 3$ is NP-complete.
    
     
\end{abstract}



\section{Introduction}

The \textit{chromatic number} $\chi(G)$ of a graph $G$ is a well-known and thoroughly studied graph parameter.
A closely related parameter is the \textit{circular chromatic number} $\chi_c(G)$ of a graph $G$, which is defined
in terms of \textit{circular colourings}. Given $r\in \mathbb R^+$, a \textit{circular $r$-colouring} of a graph
$G$ is a function $f\colon V(G)\to S^1$ from the vertices of $G$ to the unit circle such that for every
edge $xy$ of $G$ both circular arcs defined by $f(x)$ and $f(y)$ have length at least $1/r$,
and
\[
\chi_c(G) = \inf\{r\colon \text{ there is a circular } r\text{-colouring of } G\}.
\]
It is straightforward to notice that every $k$-colourable graph admits a circular $k$-colouring,
so $\chi_c(G) \le \chi(G)$. Moreover, Zhu~\cite[Theorem 1.1]{zhuDM229} proved that these parameters mutually
bound each other by the inequalities $$\chi(G) -1 < \chi_c(G) \le \chi(G).$$
It is then
natural to ask for which graphs $G$ the equality $\chi_c(G) = \chi(G)$
holds, or, phrased differently,  for which graphs $G$ the strict inequality $\chi_c(G)
< \chi(G)$ holds.  This question was first posed by Vince~\cite{vinceJGT12}  and
investigated in~\cite{guichardJGT17,hatamiJGT47}. It was proved by
Guichard~\cite{guichardJGT17} that
testing whether $\chi(G) = \chi_c(G)$ is NP-hard, and later
Hatami and Tuserkani~\cite{hatamiJGT47} proved that this problem remains
NP-hard even if the chromatic number of $G$ is known. It also follows from their
proof  that deciding whether $\chi_c(G) < 4$ is NP-complete (see e.g.,~\cite[Theorem 21]{guzmanAMC438}),
but leaves open the complexity of testing $\chi_c(G) < n$ for other integers $n$. In particular, it
was asked in~\cite{guzmanAMC438} to determine the complexity of this problem in the case $n = 3$.

In spite of the NP-hardness of testing $\chi_c(G) < \chi(G)$, there are some interesting
characterizations of the graph classes defined by $\chi_c(G) < n$ for a fixed positive integer $n$. 
For instance, in~\cite{guzmanAMC438}, the authors characterize these graph classes
by means of forbidden circular orderings. Independently, and in a different
context, Brandt~\cite{brandtCPC8} proved the following characterization of graphs
$G$ with $\chi_c(G) < 3$ --- we state it  as in Theorem 3.5 from~\cite{zhuDM229},
and in Figure~\ref{fig:Petersen} we depict the Petersen graph $H_{10}$, and the
Petersen graph minus a vertex $H_{10}-v$.

\begin{theorem}\cite{brandtCPC8}\label{thm:brandt}
    The circular chromatic number of a graph $G$ is strictly less than $3$ if and only
    if there is a maximal triangle-free supergraph $G'$ of $G$ that does not contain
    $H_{10}- v$ as subgraph.
\end{theorem}

\begin{figure}[ht!]
\centering
\begin{tikzpicture}[scale = 0.8]

  \begin{scope}
    \node [vertex] (1) at (90:1) {};
    \node [vertex] (2) at (162:1) {};
    \node [vertex] (3) at (234:1) {};
    \node [vertex] (4) at (306:1) {};
    \node [vertex] (5) at (18:1) {};

    \node [vertex] (a) at (90:2) {};
    \node [vertex] (b) at (162:2) {};
    \node [vertex] (c) at (234:2) {};
    \node [vertex] (d) at (306:2) {};
    \node [vertex] (e) at (18:2) {};

    \foreach \from/\to in {1/3, 3/5, 5/2, 2/4, 4/1,
    a/b, b/c, c/d, d/e, e/a, 1/a, 2/b, 3/c, 4/d, 5/e}     
    \draw [edge] (\from) to (\to);
    \node (L1) at (0,-2.2) {$H_{10}$};
    
  \end{scope}
  
  \begin{scope}[xshift=8cm]
    \node (L2) at (0,-2.2)  {$H_{10}-v$};
    \node [vertex] (1) at (0,0) {};
    \node [vertex] (0) at (-1.25,0) {};
    \node [vertex] (2) at (1.25,0) {};

    \node [vertex] (11) at (0,1.7) {};
    \node [vertex] (00) at (-1.25,1.7) {};
    \node [vertex] (22) at (1.25,1.7) {};

    \node [vertex] (b) at (0,-1.7) {};
    \node [vertex] (a) at (-1.25,-1.7) {};
    \node [vertex] (c) at (1.25,-1.7) {};
    
    \foreach \from/\to in {0/00, 0/a, 1/11, 1/b, 2/22, 2/c,
    00/b, 11/a, 11/c, 22/b}     
    \draw [edge] (\from) to (\to);

    \draw [edge] (c) to [bend right = 15] (00);
    \draw [edge] (22) to [bend right = 15] (a);

  \end{scope}

\end{tikzpicture}
\caption{To the left, the Petersen graph $H_{10}$, and to the right, the graph $H_{10}-v$.}
\label{fig:Petersen}
\end{figure}
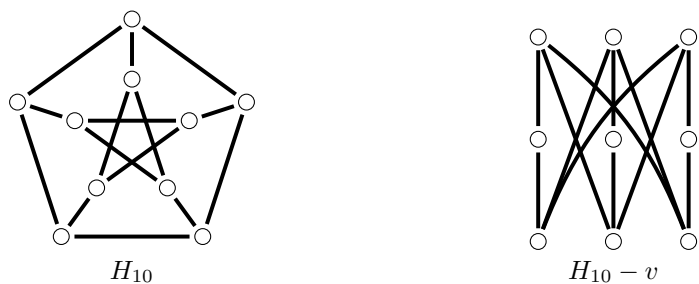

According to Zhu~\cite{zhuDM229}, the characterization of graphs with circular
chromatic number strictly less than $3$  from Theorem~\ref{thm:brandt} seems mysterious,
and it is not clear whether it can be extended to graphs with circular chromatic number
strictly less that $n$ for arbitrary integers $n$.

It follows form the definition of $\chi_c$ that a graph $G$ has circular chromatic number strictly less than
$3$ if and only if there is function $f\colon V(G) \to S^1$ such that for every edge $xy\in E(G)$ 
both circular arcs with end-points $f(x)$ and $f(y)$ have length strictly larger that $1/3$;
equivalently, both angles between $f(x)$ and $f(y)$ are strictly larger than $2\pi/3$. 

This  inspires the definition of the \textit{generic circular triangle-free graph} 
$\mathbb C_3$. The vertex set of $\mathbb C_3$ is a countable dense subset of the unit circle $S^1 \subseteq {\mathbb R}^2$
such that any two distinct points lie at
a rational angle and there is an edge $xy$ if and only if  both circular arcs with end-points
$x$ and $y$ have length strictly larger that $1/3$ (any choice of such a subset leads to the same graph,
up to isomorphism; this will follow from Theorem~\ref{thm:ax}).
Motivated by determining the complexity of testing $\chi_c(G) < 3$~\cite[Question 35]{guzmanAMC438} we look at
$\mathbb C_3$ from a constraint satisfaction theory point of view, which in turn leads us to investigate
its model theoretic properties. Moreover, we study  the generic circular triangle-free graph from a finite
graph theory perspective, and from this we will reprove and shed some light into Brandt's characterization
of $\chi_c(G) <3$ (Theorem~\ref{thm:brandt}). 

\medskip 
\textbf{Contributions.} In this article we show that 
a finite graph $G$ embeds into $\mathbb C_3$ if and only if
it is $\{K_3, K_1 + 2K_2, K_1+C_5, C_6\}$-free (see Figure~\ref{fig:minimal-obstructions}). Equivalently, 
a finite graph $G$ is $\{K_3, K_1 + 2K_2, K_1+C_5, C_6\}$-free if and only if there is a representation
of $V(G)$ by points on a circle such that two vertices are adjacent if and  only
if their corresponding points lie at an angle larger that $2\pi/3$.  By taking graph complementation, we obtain as a byproduct
of this result that a graph $G$ is a unit Helly circular-arc graph if and only if it is 
$\{3K_1, W_4, W_5, \overline{C_6}\}$-free (Theorem~\ref{thm:UCA}). As a second application of this
result, we see that a graph $G$ has circular chromatic number strictly less than three if and only if
there is a $\{K_3, K_1 + 2K_2, K_1+C_5, C_6\}$-free spanning supergraph of $G$, i.e., there is some
edge set $E'\supseteq E(G)$ such that $(V(G), E')$ is  $\{K_3, K_1 + 2K_2, K_1+C_5, C_6\}$-free. In 
Section~\ref{sec:brandt} we observe that this
is equivalent to Brandt's characterization of $\chi_c(G) < 3$ (Theorem~\ref{thm:brandt}), but highlight
that our characterization has a natural generalization to larger integers $n$. 
\begin{figure}[ht!]
\centering
\begin{tikzpicture}[scale = 0.8]

  \begin{scope}
    \node [vertex] (1) at (90:1) {};
    \node [vertex] (2) at (210:1) {};
    \node [vertex] (3) at (330:1) {};

    \foreach \from/\to in {1/2, 2/3, 1/3}     
    \draw [edge] (\from) to (\to);
    \node (L1) at (0,-1.5) {$K_3$};
    
  \end{scope}
  
  \begin{scope}[xshift=9.5cm]
    \node (L2) at (0.5,-1.5)  {$K_1 + C_5$};
    \node [vertex] (0) at (90:1) {};
    \node [vertex] (1) at (162:1) {};
    \node [vertex] (2) at (234:1) {};
    \node [vertex] (3) at (306:1) {};
    \node [vertex] (4) at (18:1) {};

    \node [vertex]  at (0:2) {};

    \foreach \from/\to in {1/2, 2/3, 3/4, 0/1, 0/4}     
    \draw [edge] (\from) to (\to);

  \end{scope}
  
  \begin{scope}[xshift=15cm]
    \node (L2) at (0,-1.5)  {$C_6$};
    \node [vertex] (0) at (60:1) {};
    \node [vertex] (1) at (120:1) {};
    \node [vertex] (2) at (180:1) {};
    \node [vertex] (3) at (240:1) {};
    \node [vertex] (4) at (300:1) {};
    \node [vertex]  (5) at (0:1) {};

    \foreach \from/\to in {1/2, 2/3, 3/4, 0/1, 4/5, 5/0}     
    \draw [edge] (\from) to (\to);
    
  \end{scope}

  \begin{scope}[xshift=5cm]
    \node (L2) at (0,-1.5) {$K_1  + 2K_2$};
    \node [vertex] (1) at (90:0.86) {};
    \node [vertex] (2) at (180:0.5) {};
    \node [vertex] (3) at (240:1) {};
    \node [vertex] (4) at (300:1) {};
    \node [vertex]  (5) at (0:0.5) {};

    \foreach \from/\to in {3/4, 2/5}     
    \draw [edge] (\from) to (\to);
   
  \end{scope}
 
\end{tikzpicture}
\caption{Four graphs that do not embed into  $\mathbb C_3$ (see Lemma~\ref{lem:Age(C3)->F-free}).}
\label{fig:minimal-obstructions}
\end{figure}
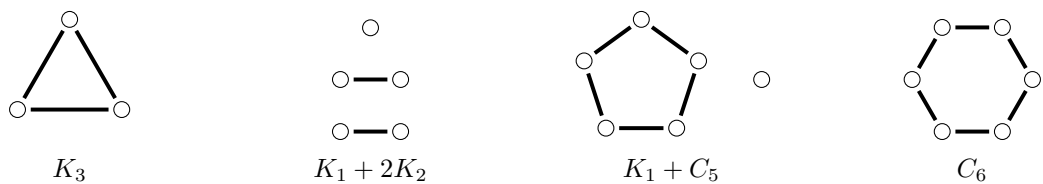

We then move on to the model theoretic properties of the generic circular triangle-free graph (Section~\ref{sec:expansion}).
In particular, we propose a (finite!) axiomatization of its first-order theory (Corollary~\ref{cor:ax}). We obtain this
axiomatization by equipping $\mathbb C_3$ with relations $B,S$, and axiomatizing the first-order theory
of the expansion $(\mathbb C_3, B,S)$ (Theorem~\ref{thm:ax}) --- actually, $S$ is the well-known separation relation
on the circle, and $B$ is closely related to the betweenness relation on the rationals. Moreover, we show that 
$(\mathbb C_3, B,S)$ is homogeneous, i.e.,  every isomorphism between finite substructures of
$(\mathbb C_3,B,S)$ extends to an automorphism of $(\mathbb C_3,B,S)$. As an applications of these results
we see that every \emph{countable} $\{K_3, K_1 + 2K_2, K_1+C_5, C_6\}$-free graph embeds into $\mathbb C_3$ (extending
our previously mentioned result for finite graphs), and that $\mathbb C_3$ is a model-complete core (Corollary~\ref{cor:mt}).

Finally, in Section~\ref{sec:complexity} we show that the problem of deciding $\chi_c(G) < 3$ for a given graph $G$
is NP-complete; equivalently, it is NP-complete to decide whether $G\to \mathbb C_3$. We first prove this by reduction
from a promise constraint satisfaction problem (closely related to approximate colourings in graph theory) 
which is known to be NP-hard. Then we propose an alternative proof based on a pp-construction of $K_3$ in
$\mathbb C_3$ (corresponding to a gadget reduction from $3$-colourability). We close this section with 
two open problems  about the computational complexity of $\Csp(G)$ for infinite graphs $G$, and a question
about graph classes with bounded chromatic number. These are motivated from wide-open conjectures in 
promise constraint satisfaction theory~\cite[Conjecture 1.2]{BrakensiekGuruswami18}
and in infinite domain CSPs~\cite[Conjecture 1.2]{BPP-projective-homomorphisms}.

Most of the elementary notions  needed for this work will be introduced in 
Section~\ref{sec:preliminaries}, except for some notions needed locally. In particular,
background about  circular-arc graphs is introduced until Section~\ref{sec:UCA-graphs},
and constraint satisfaction theory tools are introduced in Section~\ref{sec:complexity}.

\section{Preliminaries}
\label{sec:preliminaries}

In this section we fix some terminology and notation. We follow standard practise in graph theory
and model theory so that this section can be  safely skipped by many readers. 

\subsection{Graphs}

All graphs considered in this work are simple and undirected, and with (possibly infinite)
countable but non-empty vertex set.  For standard graph theoretical  notions we refer the
reader to~\cite{bondy2008}. In particular, for a positive integer $k$ we denote by
$K_k$ the complete graph on $k$ vertices, and we refer to $K_3$ as the \textit{triangle}.
A set of vertices $U$ of a graph $G$ is called an \textit{independent set} if there are no edges
$xy\in E$ with $x,y\in U$.


If $G$ is a graph, then $V(G)$ denotes the vertices of $G$, and $E(G)$ the edges of $G$.
The \textit{complement} of $G$ is the graph $\overline G$ with same vertex set as $G$
where two distinct vertices are adjacent if and only if they are not adjacent in $G$.
An \textit{isolated} vertex in a graph $G$ is a vertex $v$ not adjacent to any other vertex in $G$.
Complementarily, a \textit{universal} vertex  in $G$ is a vertex $v$ adjacent to all vertices
in $V(G)\setminus \{v\}$. Given graphs $G$ and $H$ with disjoint vertex sets we denote by $G + H$
their disjoint union, i.e., the graph with vertex set $V(G) \cup V(H)$ and edge set $E(G) \cup E(H)$.

A \textit{supergraph} of a graph $G$ is a graph $H$ such that $V(G)\subseteq V(H)$ and $E(G)\subseteq E(H)$.
In this case, we also say that $G$ is a \textit{subgraph} of $H$, and if $G\neq H$  we say that $G$ is
a \textit{proper} subgraph of $H$ (and that $H$ is a \textit{proper} supergraph of $G$). Also,
if $V(G) = V(H)$ we say that $H$ is a \textit{spanning supergraph} of $G$, and that $G$ is a
\textit{spanning subgraph} of $H$.  We say that a graph $G$ is \textit{maximal} $\mathcal{F}$-free if any 
proper spanning supergraph $H$ of $G$ is not $\mathcal F$-free 
Otherwise, we say that $H$ is an $\mathcal F$-free \textit{extension} of $G$, and
that $G$ \textit{can be extended} to an $\mathcal F$-free graph. For instance, $C_6$ can be extended to 
a triangle-free graph. Actually, a graph $G$ is maximal triangle-free if and only it is triangle-free and
if $a,b \in V(G)$ are non-adjacent, then they have a common neighbour --- this family of graphs has been
studied in~\cite{pachDM37}, and will be useful for this work. 

A \textit{homomorphism} $f\colon G\to H$ from a (possibly infinite)  graph $G$ to a
(possibly infinite) graph $H$ is an edge preserving function $f\colon V(G)\to V(H)$,  i.e., 
$f(x)f(y)\in E(H)$ whenever $xy\in E(G)$. If such homomorphism exists we write $G\to H$,
and $G\not\to H$ otherwise. A \textit{full-homomorphism} is a homomorphism
$f\colon G\to H$ such that $xy\in E(G)$ if and only if $f(x)f(y) \in E(G)$ (some authors use the term
\textit{strong homomorphism}). A full-homomorphisms that is also injective is called an \textit{embedding},
and in turn, a surjective embedding is called an \textit{isomorphism}.  If there is an isomorphism
$f\colon G\to H$ we say that $G$ and $H$ are \textit{isomorphic} and denote this fact by $G\cong H$.
Finally, an \textit{automorphism (of $G$)} is an isomorphism $f\colon G\to G$.

In graph theoretic terms, $G$ embeds into $H$ if and only if $G$ is an \textit{induced subgraph} of $H$.
Given a set of graphs $\mathcal F$ we say that $H$ is \textit{$\mathcal F$-free}
if no graph $F\in \mathcal F$ embeds into $H$; when $\mathcal F = \{F\}$ we simply
write $F$-free. In particular, a graph $G$ is triangle-free if and only if $G$ does not
contain a complete subgraph on three vertices.

A class of graphs $\mathcal C$ is \textit{hereditary} if it is preserved by vertex deletions, equivalently, 
if $G \in \mathcal C$ and $H$ embeds into $G$, then $H \in \mathcal C$. A \textit{minimal obstruction} 
of a hereditary class $\mathcal C$ is a graph $G\not \in \mathcal C$ such that every graph $G'$ obtained by deleting
one vertex from $G$ belongs to $\mathcal C$. For example, odd cycles are the minimal obstructions of bipartite
graphs. 

\subsection{Circular chromatic number}
The \textit{chromatic number} $\chi(G)$ of a graph $G$ can be defined in terms of homomorphisms by
$\chi(G) := \min\{k\colon G\to K_k\}$. We already defined the circular chromatic number in terms
of circular colourings, and now we also define it in terms of homomorphisms as follows (to see that these
are equivalent definitions we refer the reader to \cite{zhuDM229}). 
Given positive integers $p,q$ that are relatively prime, the \textit{circular complete graph} $K_{p,q}$ is the graph with
vertex set $\{0,1\dots, p-1\}$
which has an edge $ij$ if and only if $q \le |i-j| \le p-q$.
In particular, if $q = 1$,  then $K_{p,q}$ is isomorphic to the complete graph $K_p$, and if $q = 2$,
then $K_{p,q}$ is isomorphic to the complement of the odd cycle $C_p$. Also note that $K_{2n+1,n}$ is
isomorphic to the odd cycle $C_{2n+1}$. We provide an illustration of such cases in  Figure~\ref{fig:examples}.

\begin{figure}[ht!]
\centering
\begin{tikzpicture}[scale = 0.8]

  \begin{scope}
      \foreach \i in {0,1,2,3,4,5,6}
       \node [vertex, label =90-\i*360/7:{$\i$}] (\i) at (90-\i*360/7:1.3) {};

    \foreach \from/\to in {0/1, 0/2, 0/3, 0/4, 0/5, 0/6, 2/3, 2/4, 2/5, 2/6,
    1/2, 1/3, 1/4, 1/5, 1/6, 3/4, 3/5, 3/6, 4/5, 4/6, 5/6}     
    \draw [edge] (\from) to (\to);
    \node (L1) at (0,-2.3) {$K_{7,1} \cong K_7$};
  \end{scope}
  
  \begin{scope}[xshift=5cm]
      \foreach \i in {0,1,2,3,4,5,6}
        \node [vertex, label =90-\i*360/7:{$\i$}] (\i) at (90-\i*360/7:1.3) {};

    \foreach \from/\to in {0/2, 0/3, 0/4, 0/5,  2/4, 2/5, 2/6, 3/5, 3/6,
    1/3, 1/4, 1/5, 1/6, 4/6}  
     \draw [edge] (\from) to (\to);
    \node (L1) at (0,-2.3) {$K_{7,2} \cong \overline{C_7}$};
  \end{scope}

    \begin{scope}[xshift=10cm]
    \foreach \i in {0,1,2,3,4,5,6}
        \node [vertex, label =90-\i*360/7:{$\i$}] (\i) at (90-\i*360/7:1.3) {};

    \foreach \from/\to in {0/4, 4/1, 1/5, 5/2, 2/6, 6/3, 3/0}     
    \draw [edge] (\from) to (\to);
    \node (L1) at (0,-2.3) {$K_{7,3} \cong C_7$};
  \end{scope}
\end{tikzpicture}
\caption{Three circular complete graphs with vertex $p = 7$.}
\label{fig:examples}
\end{figure}
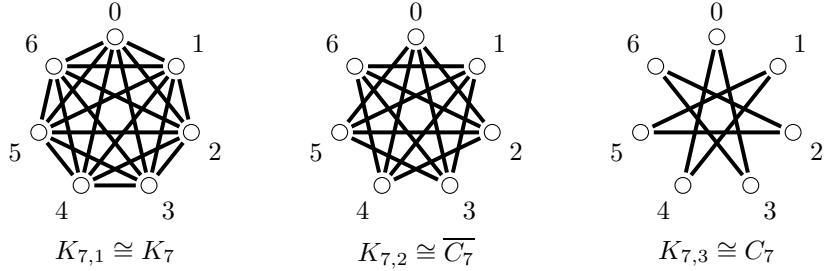

Similarly to the classical family of complete graphs,  it is the case that
$K_{p,q}\to K_{a,b}$ if and only if $p/q \le a/b$ (see e.g.,~\cite[Theorem 6.3]{HNBook}).
The circular chromatic number of a graph $G$ is then defined as
\[
\chi_c(G): = \inf\{p/q\colon p,q \text{ relative primes, } G\to K_{p,q}\}.
\]
It turns out that in the finite case, this infimum is always reached, i.e., for every
finite graph $G$ the equality $\chi_c(G) = \min\{p/q\colon p,q$ relative primes,  $G\to K_{p,q}\}$
holds (see, e.g.,~\cite[Corollary 6.8]{HNBook}).

These are all preliminary notions needed to go through Sections~\ref{sec:C3}--\ref{sec:brandt}.
The remaining preliminaries are dedicated to introduce model theoretic notions needed for
Section~\ref{sec:expansion}.

\subsection{Structures}
A (relational) \textit{signature} $\tau$ is a set of relation symbols $R,S,\dots$ each equipped with
an arity $k \in \mathbb Z^+$. A $\tau$-structure $\mathbb A$ consists of a vertex set $V(\mathbb A)$
and for each $R\in \tau$ of arity $k$ a relation $R(\mathbb A)\subseteq V(\mathbb A)^k$, which we
call the \textit{interpretation} of $R$ in $\mathbb A$. In this setting, we think of a graph $G$ as 
an $\{E\}$-structure where $E$ is an irreflexive symmetric binary relation on $V(G)$.
Homomorphisms, embeddings,  isomorphisms, and automorphisms are defined analogously
to the graph case. For a $\tau$-structure $\mathbb A$ we write $\Csp(\mathbb A)$ to denote the
class of finite $\tau$-structures $\mathbb B$ that map homomorphically to $\bA$. For instance, 
$\Csp(K_3)$ is essentially the class of $3$-colourale finite graphs --- formally, $\Csp(K_3)$ includes
all $3$-colourable finite digraphs. Similarly, we write $\Age(\bA)$ to denote the class of finite
$\tau$-structures that embed into $\bA$, and we call it the \textit{age} of $\bA$. For instance,
if $G$ is the graph with vertex set $\mathbb Z$ and edges $xy$ with $x \in \{y-1,y+1\}$, then the age of $G$
is the class of \textit{linear forests}, i.e., disjoint unions of paths. A $\tau$-structure $\mathbb A$
is \textit{finitely bounded} if there is a finite set of $\tau$-structures $\mathcal F$ such that
$\bB\in \Age(\bA)$ if and only if no structure $\mathbb F\in \mathcal F$ embeds into $\bB$.
In graph theoretic terms, this corresponds to $\Age(G)$ having finitely many minimal obstructions.

We say that a $\tau$-structure $\bB$ is a \textit{substructure} of a $\tau$-structure $\bA$ if
$V(\bB)\subseteq V(\bA)$ and for each $R\in \tau$ the interpretation $R(\bB)$ is the restriction of
$R(\bA)$ to $V(\bB)$. Notice that in the graph theoretic context, this corresponds to the notion
of \textit{induced subgraphs} (and not to subgraphs). A structure $\bA$ is \textit{homogeneous} if
for every isomorphism  $f\colon \bB\to \mathbb C$ between finite subtructures of $\bA$ there is an
isomorphism $f'\colon \bA\to\bA$ such that $f'(b) = f(b)$ for every $b\in V(\bB)$. 

Given signatures $\tau$ and $\sigma$ with $\tau \subseteq \sigma$ we say that a $\sigma$-structure
$\bB$ is an \textit{expansion} of a $\tau$-structure $\bA$ if the interpretations $R(\bA)$ and $R(\bB)$ 
coincide for every $R\in \tau$.  In this case we also say that $\bA$ is the $\tau$\textit{-reduct} of $\bB$
or simply a \textit{reduct} of $\bB$. As we will see in Section~\ref{sec:complexity}, reducts of finitely 
bounded homogeneous structures play an important role in the theory of infinite domain constraint 
satisfaction problems.

\subsection{Logic} 

The infinite structure ${\mathbb C_3}$ has the remarkable property that it is uniquely described (up to isomorphism) by the fact that it is countable and satisfies a finite set of properties that can all be expressed in first-order logic. 
Even without the fact that \emph{finitely} many properties suffice, being described by first-order properties is already a quite rare property, which has many consequences, even outside of model theory, as we will see.
This is why we need some basics from model theory.

If $\tau$ is a relational signature, then a \emph{first-order $\tau$-formula}
is defined recursively, starting from atomic formulas,   using 
the Boolean connectives $\wedge$, $\vee$, and $\neg$, universal quantification $\forall$, existential quantification $\exists$,
as usual. The atomic formulas are of the form $R(x_1,\dots,x_k)$, for $R \in \tau$ of arity $k$, and variable symbols $x_1,\dots,x_k$,
or of the form $x=y$ where $=$ is the symbol for equality and $x,y$ are variables. 
Variables that are not bounded by quantifiers in $\phi$ are called free variables of $\phi$;
we write $\phi(x_1,\dots,x_n)$ if the free variables of $\phi$ come from $x_1,\dots,x_n$. 
A $\tau$-formula without free variables is called a \emph{$\tau$-sentence}.
If $\phi$ is a $\tau$-sentence, and $\bA$ is a $\tau$-structure, then we write
$\bA \models \phi$ if $\bA$ satisfies $\phi$ (we also say that $\bA$ is a \emph{model} of $\phi$, i.e., $\phi$ is true when evaluated in $\bA$; for details, we refer to any text book in logic or model theory). 
The \emph{first-order theory of $\bA$} is the set of all first-order sentences that are satisfied by $\bA$, and denoted by $\Th(\bA)$. By an \emph{axiomatisation} of $\Th(\bA)$ we understand a set $\Phi$ of first-order $\tau$-sentences such that $\Phi$ and
$\Th(\bA)$ have the same models; i.e., every $\tau$-structure satisfies all sentences of $\Phi$ if and only if it satisfies all sentences of $\Th(\bA)$. 

A first-order formula is called \emph{existential} if it does not contain universal quantifiers and all negation symbols are in front of atomic formulas. 
Dually, a first-order formula is called \emph{universal} if it does not contain existential  quantifiers and all negation symbols are in front of atomic formulas. A first-order theory is called existential (universal) if all its sentences are existential (universal). If $\phi(x_1,\dots,x_k)$ is a first-order $\tau$-formula and $\bA$ is a $\tau$-structure, then we say that $\phi(x_1,\dots,x_k)$ \emph{defines} the relation $\{(a_1,\dots,a_k) \in V(\bA)^k \mid \bA \models \phi(a_1,\dots,a_k) \}$ 
where $\bA \models \phi(a_1,\dots,a_k)$ means that $\bA$ satisfies $\phi$ after instantiating the variables $x_1,\dots,x_k$ with the elements $a_1,\dots,a_k$; again we have to refer to any textbook in logic or model theory for the details.

\section{Circular triangle-free graphs}
\label{sec:C3}

We say that a graph $G$ is a \textit{circular triangle-free graph} if every vertex $x$
of $G$ can be represented by a point $p_x$ of a circle in such a way that $xy\in E(G)$
if and  only if $p_x$ and $p_y$ lie at an angle larger that $2\pi/3$ --- notice that $K_3$
does not admit such a representation, so every circular triangle-free graph is triangle-free.

Through this work, we use ``$G$ is circular triangle-free'' and ``$G$ embeds into $\mathbb C_3$''
equivalently (which indeed are equivalent statement from the corresponding definitions).
We begin by noticing that neither of the graphs depicted in
Figure~\ref{fig:minimal-obstructions} are circular triangle-free graphs, and the main result of
this section asserts that a finite graph $G$ is circular triangle-free if and only if 
it is $\{K_3,~K_1 + 2K_2,~K_1+C_5,~C_6\}$-free.

\begin{lemma}\label{lem:Age(C3)->F-free}
    Every circular triangle-free graph $G$ is 
    $\{K_3,~K_1+2K_2,~K_1+C_5,~C_6\}$-free.
\end{lemma}
\begin{proof}
It suffices to prove that neither of $K_3,~K_1+2K_2,~K_1+C_5,~C_6$ are circular triangle-free
graphs, and $K_3$ is clearly not. We provide an idea of why $K_1+2K_2$, $K_1+C_5$, and 
$C_6$ are not circular triangle-free graphs; the details are left to the reader. 
Fix embeddings of $2K_2$, $C_5$, and $P_5$ into $\mathbb C_3$; such embeddings are depicted below,
where the non-filled vertices indicate the image of the embedding, and the two dashed edges 
represent that at least one of them must be an edge of $\mathbb C_3$ (this will be argued shortly).

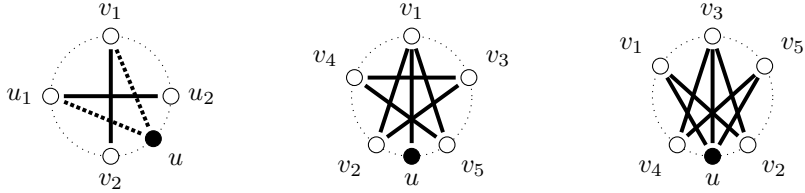
\begin{center}
    \begin{tikzpicture}[scale = 0.8]

    \begin{scope}
        \draw[dotted] (0,0) circle (1);
        \node [vertex, label = above:{$v_1$}] (v1) at (90:1) {};
        \node [vertex, label = below:{$v_2$}] (v2) at (-90:1) {};
        \node [vertex, label = left:{$u_1$}] (u1) at (180:1) {};
        \node [vertex, label = right:{$u_2$}] (u2) at (0:1) {};
        \node [vertex, fill = black, label = -45:{$u$}] (u) at (-45:1) {};

        \foreach \from/\to in {v1/v2, u1/u2}     
        \draw [edge] (\from) to (\to);
        \draw [edge, densely dotted] (u) to (v1);
        \draw [edge, densely dotted] (u) to (u1);
    \end{scope}

    \begin{scope}[xshift = 5cm]
        \draw[dotted] (0,0) circle (1);
        \node [vertex, label = 90:{$v_1$}] (1) at (90:1) {};
        \node [vertex, label = 162:{$v_4$}] (4) at (162:1) {};
        \node [vertex, label = 234:{$v_2$}] (2) at (234:1) {};
        \node [vertex, label = 306:{$v_5$}] (5) at (306:1) {};
        \node [vertex, label = 18:{$v_3$}] (3) at (18:1) {};
        \node [vertex, fill = black, label = below:{$u$}] (u) at (-90:1) {};

        \foreach \from/\to in {1/2, 2/3, 3/4, 4/5, 5/1, u/1}     
        \draw [edge] (\from) to (\to);  
    \end{scope}

    \begin{scope}[xshift = 10cm]
        \draw[dotted] (0,0) circle (1);
        \node [vertex, label = 90:{$v_3$}] (1) at (90:1) {};
        \node [vertex, label = 150:{$v_1$}] (4) at (150:1) {};
        \node [vertex, label = 234:{$v_4$}] (2) at (234:1) {};
        \node [vertex, label = 306:{$v_2$}] (5) at (306:1) {};
        \node [vertex, label = 30:{$v_5$}] (3) at (30:1) {};
        \node [vertex, fill = black,  label = below:{$u$}] (u) at (-90:1) {};

        \foreach \from/\to in {1/2, 2/3, 4/5, 5/1, u/4, u/3, u/1}     
        \draw [edge] (\from) to (\to);
    \end{scope}
    \end{tikzpicture}
\end{center}
It is not hard to observe that none of these partial embeddings $f\colon (G-u)\to \mathbb C_3$
for $G\in\{K_1+2K_2,~K_1+C_5,~C_6\}$ can be extended to an embedding of $G$ into $\mathbb C_3$. This observation is easiest for 
the partial embedding of $K_1 + C_5$ (middle picture), because every vertex $u\in V(\mathbb C_3)$ in the clockwise arc
from $v_5$ to $v_2$ will be adjacent to $v_1$. Symmetrically, any vertex $u$
of $\mathbb C_3$ is a neighbour of some vertex in the image of the partial embedding. In fact, the argument does not depend on the originally chosen embedding of $C_5$ into ${\mathbb C}_5$
because the embedding is unique up to applying an automorphism of ${\mathbb C}_5$.
Hence, $K_1+C_5$ cannot be embedded
into $\mathbb C_3$. 

Similarly, considering the embedding of $P_5$
into $\mathbb C_3$ (right-most picture), once can notice that any common neighbour $u$ of $v_1$ and $v_5$ must
lie in the clockwise arc from $v_2$ to $v_4$, and thus be a neighbour of $v_3$, so $C_6$ cannot be embedded
into $\mathbb C_3$.
Finally, consider the partial embedding of $2K_2$ (left-most picture) and a vertex
$u$ in the clockwise arc from $u_2$ to $v_2$. In this case, $u$ must be adjacent to either
$v_1$ or $u_1$ as otherwise, the clockwise arcs form $v_1$ to $u$, from 
$u$ to $u_1$, and from $u_1$ to $v_1$ have length strictly less than $1/3$, contradicting
the fact that the circumference is of length  $1$. After considering the symmetric cases, we conclude that
$K_1 +2K_2$ does not embed into $\mathbb C_3$.
\end{proof}

Now, we relate the graphs that embed into $\mathbb C_3$ to the graphs that embed into some $K_{3k-1,k}$
--- recall that
$K_{3k-1,k}$ has vertex set $\{0,\dots, 3k-2\}$ and there  is an edge $ij$ if and only if $|i-j| \ge k$
modulo $3k-1$. In particular, $K_{3k-1,k}\cong K_2$ for $k = 1$, and $K_{3k-1,k} \cong C_5$ for $k = 2$.

\begin{lemma}\label{lem:3k-1/3}
    The following statements are equivalent for a finite graph $G$.
    \begin{itemize}
        \item $G$ is a circular triangle-free graph.
        \item $G$ embeds into $K_{3k-1,k}$ for some positive integer $k$.
        \item $G$ admits a full-homomorphism to $K_{3k-1,k}$  for some positive integer $k$.
    \end{itemize}
\end{lemma}
\begin{proof}
    The second item is a particular case of the third one because every embedding is a full-homomorphism.
    Suppose that $G$ is as in the third statement; we show that then the first statement holds. Notice  that if
    there is a full-homomorphism $f\colon G \to H$ and an embedding $h\colon H\to \mathbb C_3$, then we
    can map each $v \in V(G)$ to a point $g(v)$ which is sufficiently close to $h(f(v))$ such that
    $g\colon G \to \mathbb C_3$ defines an embedding. Thus, in order to prove that the first statement
    holds, it suffices to argue that $K_{3k-1,k}$
    embeds into $\mathbb C_3$ for every positive integer $k$. Such an embedding can be constructed as follows. First, for each $i\in\{0,\dots, 3k-2\}$ consider the point $p_i$ defined by the angle
    $2\pi i,3k-1$. In particular, the circular distance between $p_i$ and $p_j$ is strictly larger than $1/3$ if and 
    only if $|i-j| \ge k$, i.e., if and only if $ij\in E(K_{3k-1,k})$. Second, approximate each $p_i$ by a point
    $p_i'$ such that $p_i'$ is defined by a rational, and the circular distance
    between $p'_i$ and $p'_j$ is larger than $1/3$ if and only if the circular distance
    between $p'_i$ and $p_j$ is. Hence, the mapping $i\mapsto p_i'$ defines an embedding 
    from $K_{3k-1,k}$ into $\mathbb C_3$.

    Finally we prove that the first item implies the second one.  Suppose that $G$ embeds into ${\mathbb C}_3$.
    With such an embedding $f\colon G\to \mathbb C_3$ one can describe each vertex $v$ of $G$ by a rational angle $p_v$. 
    The idea now is to choose $k$ to be large enough so that each $p_v$ is a approximated by a point of the
    form $2\pi i_v/(3k-1)$ in such a way that $i_v\neq i_u$ if $u\neq v$, and the circular distance between $p_v$ and $p_u$
    is strictly larger than $1/3$ if and only if the circular distance between $2\pi i_v/(3k-1)$ and $2\pi i_u/(3k-1)$ is.
    Technically, for each $\epsilon > 0$ we choose $k$ large enough so that $1/(3k-1) < \epsilon$
    and $1/(3k-1)$ is also smaller than half of the minimum circular distance between all $p_v,p_u$ for $u,v\in V(G)$.
    Then, by considering $\epsilon$ small enough we find the suggested approximation $2\pi i_v/(3k-1)$ of each $p_v$. 
    Notice that $2\pi i_v/(3k-1)$ is not a vertex of $\mathbb C_3$, but this irrelevant for the purpose of this 
    proof as such an approximation defines an embedding $f'\colon G\to K_{3k-1,k}$ by mapping
    $v$ to $i_v$.
\end{proof}


With this simple lemma, we can build on the following theorem to prove that the
converse of Lemma~\ref{lem:Age(C3)->F-free} holds for maximal triangle-free finite
graphs. The following theorem is an equivalent restatement of Theorem 2 in~\cite{pachDM37}.

\begin{theorem}[\hspace{1sp}\cite{pachDM37}]\label{thm:pach}
Let $G$ be any finite  triangle-free graph. 
If every independent set of vertices of $G$ has a common
neighbour, then there is a full-homomorphism $f\colon G\to K_{3k-1,k}$ for some positive 
integer $k$.
\end{theorem}

We build on this theorem together with Lemma~\ref{lem:3k-1/3} to characterize the maximal
triangle-free graphs that embed into $\mathbb C_3$ as follows.

\begin{theorem}\label{thm:AgeC3+maximal}
    The following statements are equivalent for a finite maximal triangle-free graph $G$.
    \begin{itemize}
        \item $G$ is a circular triangle-free graph.
        \item $G$ embeds into $K_{3k-1,k}$ for some positive integer $k$.
        \item $G$ admits a full-homomorphism to $K_{3k-1,k}$  for some positive integer $k$.
        \item $G$ is $\{K_3,~K_1 + 2K_2,~K_1 + C_5,~C_6\}$-free.
    \end{itemize}
\end{theorem}
\begin{proof}
    The equivalence between the first three items was proved in Lemma~\ref{lem:3k-1/3}. Here we show that the
    last statement is equivalent to the first three.  
    If $G$ satisfies the first statement, then it is $\{K_3, C_5+ K_1, C_6, 2K_2+K_1\}$-free by Lemma~\ref{lem:Age(C3)->F-free}. 
    We show that the last statement implies the third one. The fact that $G$ is a maximal-triangle free graph
    implies that any two non-adjacent vertices have a common neighbour in $G$. Moreover, we claim that 
    for each independent set $I \subseteq V(G)$ there
    is a vertex $v$ such that $I \subseteq N(v)$. To prove this, suppose 
    otherwise for contradiction, and let $I$ be a minimal counterexample.
    In particular, $|I| \ge 3$ and for each $x\in I$ there is a vertex $v_x$ which is a neighbour of every
    $y\in I\setminus \{x\}$ and $xv_x\not \in E(G)$. Let $x,y,z\in I$ be three different vertices of $I$,
    and notice that $x v_z y v_x z v_y$ induce a $6$-cycle in $G$: this assertion follows from the
    choice of $v_x$, $v_y$, $v_z$, the fact that $I$ is independent, and because $G$ has no triangles. 
    The existence of this $6$-cycle contradicts the assumption that $G$ is $C_6$-free. 
    Therefore, $G$ satisfies that every independent subset of vertices has a common neighbour.
    Thus, it follows from Theorem~\ref{thm:pach} that $G$
    admits a full-homomorphism to $K_{3k-1,k}$ for some positive integer $k$.
\end{proof}

\subsection{Vertex extension lemmas}

This subsection contains a series of technical lemmas that build up to showing
that every $\{K_3,~K_1+2K_2,~K_1+C_5,~C_6\}$-free graph $G$ can be embedded into a 
possibly larger but still finite 
$\{K_3,~K_1+2K_2,~K_1+C_5,~C_6\}$-free graph $H$ with the
extra property that $H$ is maximal triangle-free. 

\begin{lemma}\label{lem:ext-C6}
    Let $G$ be a $\{K_3,~K_1+2K_2,~C_6\}$-free graph and let $I\subseteq V(G)$ be a maximal
    independent set of vertices. If $G'$ is obtained from $G$ by adding a new vertex $v'$ 
    such that $N(v') = I$, 
    then $G'$ is $C_6$-free.
\end{lemma}
\begin{proof}
    Since $G$ is $C_6$-free, $C_6$ embeds into $G'$ if and only if $v'$ belongs to
    a $6$-cycle. The latter holds true if and only if there is an induced
    path on $5$ vertices $v_1,\dots, v_5$ such that $v_1,v_5\in I$ and
    $v_2,v_3,v_4\in V(G)\setminus I$. 
    Since $I$ is a maximal independent set in $G$, the vertex 
    $v_3$ must have a neighbour $u\in I$. Since $G$ is $K_3$-free, $u$ is neither adjacent to $v_2$ nor to $v_4$. Moreover, $u$ is neither adjacent to $v_1$ nor to $v_5$, 
    because $v_1,v_5,u\in I$. It follows that the vertices $u$, $v_1$, $v_2$, $v_4$,
    and $v_5$ induce a copy of $K_1+ 2K_2$  in $G$,  contradicting the choice of $G$.
    Therefore, $G'$ is $C_6$-free.
\end{proof}

The following lemma is a building block to proving that the same construction as in Lemma~\ref{lem:ext-C6} yields a $(K_1+2K_2)$-free graph.

\begin{lemma}\label{lem:I-K12K1}
    Let $G$ be a $\{K_3,~K_1+ 2K_2,~C_6\}$-free graph. If  $I\subseteq V(G)$ is a 
    maximal independent set of vertices, then every 
    induced $2K_2$ subgraph of $G$ contains at least one vertex in $I$.
\end{lemma}
\begin{proof}
    Anticipating a contradiction, suppose that there are four vertices $a,b,u,v\in 
    V(G)\setminus I$
    that induce exactly two edges $ab$ and $uv$. We first argue that there is a vertex
    $x\in I$ adjacent to exactly one vertex of each edge $ab$ and $uv$. Indeed, 
    since $I$ is a maximal independent set, there are vertices $a',b',u',$ and $v'$
    adjacent to $a,b,u,$ and $v$, respectively. If all of these are different, 
    and each $i\in \{a,b,u,v\}$ is only adjacent to $i'$, then it is straightforward to
    find a copy of $K_1 + 2K_2$ in $G$. Thus, some vertex $x\in \{a',b',u',v'\}$ must be
    adjacent to two vertices in $\{a,b,u,v\}$, and since $G$ is triangle-free,
    $x$ is adjacent to exactly one end-vertex of each edge $ab$ and $uv$. 
    Without loss of generality, we assume that $x$ is adjacent to $a$ and $u$. 
    Now we argue that there is also a vertex $y$ adjacent to $b$ and $v$. Otherwise, we have $b' \neq v'$  and 
    then $\{x, b', b, v', v\}$ induces a $K_1+2K_2$ in $G$. Finally, by the assumption
    that $ab$ and $uv$ induce a copy of $2K_2$ in $G$, and by the choice of $x$
    and $y$, we conclude that $a,b,y,v,u,x$ create an induced $6$-cycle in $G$, 
    reaching our final contradiction. Therefore, every induced $2K_2$ contains
    at least one vertex in $I$.
\end{proof}

\begin{lemma}\label{lem:ext-K12K2}
    Let $G$ be a $\{K_3,~K_1+ 2K_2,~K_1 + C_5,~C_6\}$-free graph, and $I\subseteq V(G)$ a 
    maximal independent set of vertices. If $G'$ is obtained from $G$ by adding a new vertex 
    $v'$ adjacent to exactly the vertices in $I$, then $G'$ is $(K_1 + 2K_2)$-free.
\end{lemma}
\begin{proof}
    It suffices to show that $v'$ does not belong to any copy of $K_1+2K_2$ in $G'$.
    As before, we proceed by contradiction and assume that $v'$ does belong to 
    such a copy in $G'$. There are two possible cases: $v'$ is the isolated vertex,
    or $v'$ is the end-vertex of one the edges of $K_1 + 2K_2$. In the first case,
    there would be a copy of $2K_2$ that does not intersect $I$, and this contradicts
    Lemma~\ref{lem:I-K12K1}. For the second case, let $v', u, a, b, x$
    be the vertices in $G'$ that induce a copy of $K_1+2K_2$, where $uv',ab\in E(G')$
    and $x$ is an isolated vertex (in this copy of $K_1+2K_2$). Clearly, $u$ belongs to
    $I$ while $a$, $b$, and $x$
    belong to $V(G)\setminus I$. Since $I$ is a maximal independent set, $x$ has
    a neighbour $y \in I$. Notice that if $y$ is not adjacent to neither $a$ nor $b$, 
    then $u$, $ab$, $xy$ induce a copy of $K_1+2K_1$ in $G$, so we assume without loss of
    generality that
    $by \in E(G)$. Since $G$ is $K_3$-free, $y$ is not adjacent to $a$. By the maximality
    of $I$ there is a neighbour $a'$ of $a$ that belongs to $I$. This yields 
    the following picture, where $u$ is only adjacent to $v'$.
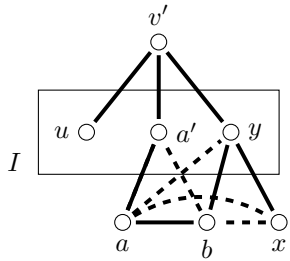
\begin{center}
    \begin{tikzpicture}[scale = 0.8]

    \begin{scope}[xshift=5cm]
        \node [vertex, label = above:{$v'$}] (v) at (0,3) {};
        
        \node [vertex, label = left:{$u$}] (u) at (-1.2,1.5) {};
        \node [vertex, label = right:{$a'$}] (a1) at (0,1.5) {};
        \node [vertex, label = right:{$y$}]  (y) at (1.2,1.5) {};
        \draw (-2,0.8) rectangle (2,2.2);
        \node at (-2.4,1) {$I$};

        \node [vertex, label = below:{$a$}] (a) at (-0.6,0) {};
        \node [vertex, label = below:{$b$}] (b) at (0.8,0) {};
        \node [vertex, label = below:{$x$}]  (x) at (2,0) {};

        \foreach \from/\to in {v/u, v/a1, v/y, y/x, y/b, b/a, a/a1} 
        \draw [edge] (\from) to (\to);

        \foreach \from/\to in {x/b, a/y, a/a1, b/a1} 
        \draw [edge, dashed] (\from) to (\to);
        \draw [edge, dashed] (a) to [bend left = 30] (x);
    \end{scope}
    \end{tikzpicture}
\end{center}
The only possible adjacency left to determine in the configuration above is 
between $a'$ and $x$. If there is an edge $a'x$ in $G$, then $u$ together with
$a, a', x, y, b$ induce a copy of $K_1 + C_5$ in $G$.  Otherwise, $\{u,a,a',x,y\}$ induces a copy of $K_1 + 2K_2$ in $G$. In either case, we obtain a contradiction to the fact that $G$ is $\{K_1+2K_2,~K_1+C_5\}$-free. 
\end{proof}

Using a similar idea as in the previous proof, we first prove the following lemma
as a first step to proving that the graph $G'$ considered in Lemma~\ref{lem:ext-C6}
and Lemma~\ref{lem:ext-K12K2} is also $(K_1+C_5)$-free.

\begin{lemma}\label{lem:I-K1C5}
    Let $G$ be a $\{K_3,~K_1+ 2K_2,~K_1 + C_5\}$-free  graph. If  $I\subseteq V(G)$ is a 
    maximal independent set of vertices, then every $5$-cycle of $G$ intersects $I$. 
\end{lemma}
\begin{proof}
    Anticipating a contradiction suppose that there are five vertices $v_1,\dots, v_5\in V(G)\setminus I$
    that induce a $5$-cycle. By the maximality of $I$, the vertices $v_1$ and $v_5$ must be adjacent to some vertices
    $u_1$ and $u_5$ in $I$, respectively. Since $G$ is $K_3$-free, $u_1$ is neither adjacent to $v_2$ nor
    to $v_5$, and $u_5$ is neither adjacent to $v_1$ nor to $v_4$. If neither $u_1$ nor $u_5$ is adjacent
    to $v_3$, then $\{v_3,u_1,v_1,u_5,v_5\}$ induces a copy of $K_1 + 2K_2$ in $G$. Thus, without loss
    of generality, we may assume that $u_1$ is adjacent to $v_3$. By the choice of $I$, the vertex $v_4$ must have a
    neighbour $u_4$ in $I$, and $u_4$ cannot be adjacent to neither $v_3$ nor $v_5$, because $G$ is $K_3$-free.
    We now distinguish between two complementary cases: one where $u_5v_3 \notin E(G)$, and the
    other one where $u_5v_3\in E(G)$. Below we draw the first configuration to the left, and the second one
    to the right (with the additional vertex $u_2$ introduced shortly). 
    
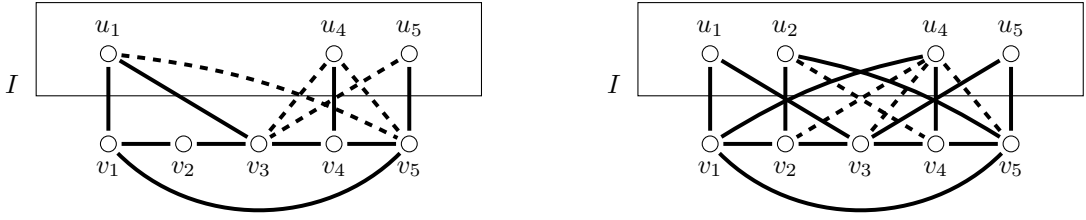
\begin{center}
    \begin{tikzpicture}[scale = 0.8]

    \begin{scope}
        \node [vertex, label = above:{$u_1$}] (u1) at (-2.5,1.5) {};
        \node [vertex, label = above:{$u_4$}] (u4) at (1.25,1.5) {};
        \node [vertex, label = above:{$u_5$}] (u5) at (2.5,1.5) {};
        \draw (-3.7,0.8) rectangle (3.7,2.35);
        \node at (-4.1,1) {$I$};

        \node [vertex, label = below:{$v_1$}] (v1) at (-2.5,0) {};
        \node [vertex, label = below:{$v_2$}] (v2) at (-1.25,0) {};
        \node [vertex, label = below:{$v_3$}] (v3) at (0,0) {};
        \node [vertex, label = below:{$v_4$}] (v4) at (1.25,0) {};
        \node [vertex, label = below:{$v_5$}] (v5) at (2.5,0) {};

        \foreach \from/\to in {v1/v2, v2/v3, v3/v4, v4/v5, u1/v1, u1/v3, u4/v4, v5/u5} 
        \draw [edge] (\from) to (\to);
        \draw [edge] (v5) to [bend left = 45] (v1);

        \foreach \from/\to in {u5/v3, u4/v3, u4/v5} 
        \draw [edge, dashed] (\from) to (\to);
        \draw [edge, dashed] (u1) to [bend left = 10] (v5);
    \end{scope}

    \begin{scope}[xshift = 10cm]
       \node [vertex, label = above:{$u_1$}] (u1) at (-2.5,1.5) {};
        \node [vertex, label = above:{$u_2$}] (u2) at (-1.25,1.5) {};
        \node [vertex, label = above:{$u_4$}] (u4) at (1.25,1.5) {};
        \node [vertex, label = above:{$u_5$}] (u5) at (2.5,1.5) {};
        \draw (-3.7,0.8) rectangle (3.7,2.35);
        \node at (-4.1,1) {$I$};

        \node [vertex, label = below:{$v_1$}] (v1) at (-2.5,0) {};
        \node [vertex, label = below:{$v_2$}] (v2) at (-1.25,0) {};
        \node [vertex, label = below:{$v_3$}] (v3) at (0,0) {};
        \node [vertex, label = below:{$v_4$}] (v4) at (1.25,0) {};
        \node [vertex, label = below:{$v_5$}] (v5) at (2.5,0) {};

        \foreach \from/\to in {v1/v2, v2/v3, v3/v4, v4/v5, u1/v1, u1/v3, u4/v4,
        v5/u5, u5/v3, u2/v2} 
        \draw [edge] (\from) to (\to);
        \draw [edge] (v5) to [bend left = 45] (v1);
        \draw [edge] (u2) to [bend left = 10] (v5);
        \draw [edge] (v1) to [bend left = 10] (u4);

        \foreach \from/\to in {u4/v3, u4/v5, u4/v2, u2/v4} 
        \draw [edge, dashed] (\from) to (\to);
    \end{scope}
    \end{tikzpicture}
\end{center}
In the first case (on the left), when $u_5v_3\not\in E(G)$ we quickly reach a contradiction
by noticing that $\{u_4,u_1,v_3,u_5,v_5\}$ induces a copy of $K_1+2K_2$ in $G$. In the
second case, when $u_5v_3\in E(G)$ we need a little more work.
First, consider the subcase when $u_4v_1\not\in E(G)$   (same picture on the left, but making 
$u_5v_3$ a solid edge and $u_4v_1$ a dashed edge). 
Then $u_4$ together with $u_1,v_3,u_5,v_5,v_1$ induce a copy of $K_1+C_5$ in $G$, which
contradicts the choice of $G$. So we assume that $u_4v_1\in E(G)$ (picture on the right).
With similar arguments
as before, there must be a neighbour $u_2$ of $v_2$ in $I$, and  since $G$ is $K_3$-free,
$v_2u_1, v_2u_4, v_2u_5\not\in E(G)$, in particular $u_2\not\in \{u_1,u_4,u_5\}$. 
For the same reason, $u_2$ is not adjacent to $v_1$ nor $v_3$. Symmetrically to some
previous arguments, if $u_2$ is not adjacent to $v_5$, then $u_2$ together with
$u_1,v_3,u_5,v_5,v_1$ induce a $K_1+ C_5$, so we assume that $u_2v_5\in E(G)$. 
Finally, notice that $u_2v_4\not\in E(G)$ because $G$ is triangle-free, hence
$\{v_1,u_2,v_2,u_4,v_4\}$ induces a copy of $K_1+2K_2$ in $G$, our final contradiction.
Therefore, $\{v_1,\dots, v_5\}\cap I\neq \varnothing$.
\end{proof}

\begin{lemma}\label{lem:ext-K1C5}
    Let $G$ be a $\{K_3,~K_1+ 2K_2,~K_1 + C_5\}$-free  graph, and $I\subseteq V(G)$ a 
    maximal independent set of vertices. If $G'$ is obtained from $G$ by adding a new vertex 
    $v'$ adjacent to exactly the vertices in $I$, then $G'$ is $(K_1 + C_5)$-free.
\end{lemma}
\begin{proof}
    Proceeding as in the previous lemmas, suppose that $v'$ belongs to
    some induced copy of $K_1 + C_5$ in $G'$. If $v'$ is the isolated vertex in this copy, 
    then there is a $5$-cycle in $G$ that does not intersect $I$, and this
    contradicts Lemma~\ref{lem:I-K1C5}. So, suppose there are vertices
    $v_2$, $v_3$, $v_4$, $v_5$ and $u$ such that $v', v_2, \dots, v_5$ is a $5$-cycle
    in $G$ and $u$ is not adjacent to any of these vertices. In particular, $u$ does not
    belong to $I$, and by the maximality of $I$ it has a neighbour $u'$ in $I$.
    This is creating the following configuration, where $u$ is only adjacent to $u'$.
    \begin{center}
    \begin{tikzpicture}[scale = 0.8]

    \begin{scope}[xshift=5cm]
        \node [vertex, label = above:{$v'$}] (v) at (0,3) {};
        
        \node [vertex, label = left:{$v_2$}] (2) at (-1.2,1.5) {};
        \node [vertex, label = right:{$v_5$}] (5) at (0,1.5) {};
        \node [vertex, label = right:{$u'$}]  (u1) at (1.2,1.5) {};
        \draw (-2,0.8) rectangle (2,2.2);
        \node at (-2.4,1) {$I$};

        \node [vertex, label = below:{$v_3$}] (3) at (-1.2,0) {};
        \node [vertex, label = below:{$v_4$}] (4) at (0,0) {};
        \node [vertex, label = below:{$u$}]  (u) at (1.2,0) {};

        \foreach \from/\to in {v/2, 2/3, 3/4, 4/5, 5/v, v/u1, u/u1} 
        \draw [edge] (\from) to (\to);

        \foreach \from/\to in {2/4, 3/5} 
        \draw [edge, dashed] (\from) to (\to);
    \end{scope}
    \end{tikzpicture}
\end{center}
    If $u'$ is not adjacent to $v_3$, then $\{v_5,v_2,v_3,u,u'\}$ induces a copy
    of $K_1 + 2K_2$ in $G$. Similarly, if $u'$ is not adjacent to $v_4$, then 
    $\{v_2,v_4,v_5,u,u'\}$ induces a copy of $K_1 + 2K_2$ in $G$. Either case 
    contradicts the fact that $G$ is $(K_1+2K_2)$-free. Hence, $u'$ is adjacent
    to both $v_3$ and $v_4$ creating a triangle in $G$, again a contradiction.
    Therefore, $G'$ is $(K_1 + C_5)$-free.
\end{proof}

Now we merge  Lemmas~\ref{lem:ext-C6}, \ref{lem:ext-K12K2}, and~\ref{lem:ext-K1C5} to obtain
the desired result.

\begin{lemma}\label{lem:maximal-v-extension}
    For every $\{K_3,~K_1+ 2K_2,~K_1 + C_5,~C_6\}$-free  graph $G$ there is a finite graph $H$
    with the following properties.
    \begin{itemize}
        \item $G$ is an induced subgraph of $H$.
        \item $H$ is  $\{K_3,~K_1+ 2K_2,~K_1 + C_5,~C_6\}$-free.
        \item $H$ is maximal triangle-free.
    \end{itemize}
\end{lemma}
\begin{proof}
    If $G$ is maximal triangle-free, take $H = G$. Otherwise, there are 
    non-adjacent vertices $x,y\in V(G)$ with no common neighbour in $G$. Extend $x,y$
    to a maximal independent set $I$ in $G$. Let $G'$ be the graph obtained from $G$
    by adding a new vertex $v'$ whose neighbours are exactly the vertices in
    $I$. Clearly, $G'$ is triangle-free. Moreover, it follows from Lemmas~\ref{lem:ext-C6},
    \ref{lem:ext-K12K2}, and~\ref{lem:ext-K1C5} that $G'$
    is  $\{K_3,~K_1+ 2K_2,~K_1 + C_5,~C_6\}$-free. Observe that $G'$
    has strictly less pairs of non-adjacent vertices without a common neighbour
    than $G$. Indeed, $v'$ is a common neighbour of $x$ and $y$, and by the maximality of $I$,
    any vertex $u$ not adjacent to $v'$ must have a neighbour in $I$, and thus a common
    neighbour with $v'$. The claim now follows inductively. 
\end{proof}

\subsection{Structural characterization of circular triangle-free graphs}

All lemmas in the previous subsection together with Theorem~\ref{thm:AgeC3+maximal}
build up to the following characterization of the age of
$\mathbb C_3$.

\begin{theorem}\label{thm:ageC3}
    The following statements are equivalent for a finite graph $G$. 
    \begin{itemize}
        \item $G$ is a circular triangle-free graph.
        \item $G$ is $\{K_3,~K_1+2K_2,~K_1+C_5,~C_6\}$-free.
        \item $G$ embeds into $K_{3k-1,k}$ for some positive integer $k$.
        \item $G$ admits a full-homomorphism to $K_{3k-1,k}$  for some positive integer $k$.
    \end{itemize}
\end{theorem}
\begin{proof}
    The first item  implies the second one (Lemma~\ref{lem:Age(C3)->F-free}). 
    Lemma~\ref{lem:3k-1/3} shows that the first statement
    is equivalent to the last two items. Thus, it suffices to prove that the second
    statement implies any of the statements 1, 3, or 4;  we show that
    it implies the first one. Suppose that $G$ is $\{K_3,~K_1+2K_2,~K_1+C_5,~C_6\}$-free,
    and let $H$ be the 
    extension of $G$ guaranteed by Lemma~\ref{lem:maximal-v-extension}, i.e.,
    $G$ embeds into a $\{K_3,~K_1+2K_2,~K_1+C_5,~C_6\}$-free maximal triangle-free graph $H$. 
    We know that $H$ embeds into $\mathbb C_3$ via Theorem~\ref{thm:AgeC3+maximal}.
    By composing embeddings, we conclude that $G$ embeds into $\mathbb C_3$.
\end{proof}

Recall that a graph $G$ has circular chromatic number strictly less that $3$ if and
only if $G\in \Csp(\mathbb C_3)$. It is not hard to see that $G\to \mathbb C_3$ if and
only if there is an injective homomorphism  $i\colon G\to \mathbb C_3$. Clearly,
$G$ can be extended to a graph $G'$ such that  $i\colon G'\to \mathbb C_3$ is
an embedding. Therefore, $\chi_c(G) < 3$ if and only if it can  be extended to
a $\{K_3,~K_1+2K_2,~K_1+C_5,~C_6\}$-free graph (Theorem~\ref{thm:ageC3}).

\begin{corollary}\label{cor:circular<3}
The following statements are equivalent for any finite graph $G$.
    \begin{itemize}
        \item $\chi_c(G) < 3$.
        \item $G\in \Csp(\mathbb C_3)$.
        \item $G$ can be extended to a $\{K_3,~K_1+2K_2,~K_1+C_5,~C_6\}$-free graph.
    \end{itemize}
\end{corollary}

\section{Unit circular-arc graphs}
\label{sec:UCA-graphs}

A \textit{circular-arc model} $\mathcal M$ of a graph $G$ is an ordered pair $(C,\mathcal A)$
where $C$ is a circle and $\mathcal A$ is a set of circular-arcs of $C$ such that $G$
is isomorphic to the intersection graph of $\mathcal A$. We say that $G$ 
is a \textit{circular-arc graph} if there is a circular-arc model $\mathcal M$ of $G$.
In this case, we say that $G$ \textit{admits} a circular-arc model. 
A \textit{proper} circular-arc model is a circular model $(C,\mathcal A)$ such that
no arc $A\in \mathcal A$ is contained in some other $B\in \mathcal A$. A circular-arc
model is a \textit{unit} circular-arc model if every arc in $\mathcal A$ has unit length, 
and a \textit{Helly} circular-arc model if $\mathcal A$ is a Helly family of sets, i.e., 
if whenever a subset of arcs $\mathcal A' \subseteq \mathcal A$  has non-empty
pairwise intersection, then the intersection of all arcs in $\mathcal A'$ is non-empty. 
We say that $G$ is a \textit{proper} (resp.\ \textit{unit}, or \textit{Helly}) circular-arc
graph if $G$ admits a proper (resp.\ unit, or Helly) circular-arc model $\mathcal M$. 
Finally, a graph $G$ is a \textit{proper Helly} (resp.\ \textit{unit Helly}) circular-arc
graph if it admits a circular-arc model which is both proper (resp.\ \textit{unit}) and Helly.
We refer the reader to~\cite{linDM309} for a survey on the class of circular-arc graphs and some
of its subclasses.

In this brief section, we use Theorem~\ref{thm:ageC3} to obtain a characterization
of $3K_1$-free unit Helly circular-arc graphs by finitely many forbidden minimal
obstructions. This characterization stems from the following immediate observation.

\begin{observation}\label{obs:immediate}
    A graph $G$ embeds into $\mathbb C_3$ if and only if $\overline G$ admits a circular-arc model
    consisting of unit closed circular-arcs of the circumference of length $3$ such that distinct  
    arcs do not have a common end point.
\end{observation}

 \begin{theorem}\label{thm:UCA}
     The following statements are equivalent for a graph $G$. 
     \begin{enumerate}
         \item $G$ is a $3K_1$-free unit Helly circular-arc graph.
         \item $G$ is $\{3K_1,~W_4,~W_5,~\overline{C_6}\}$-free. 
         \item $\overline G$ is an induced subgraph of $\mathbb C_3$.
         \item $G$ admits a unit circular-arc model consisting of closed arc in the circumference
         of length $3$ with no common end points. 
     \end{enumerate}
 \end{theorem}
\begin{proof}
    The last two items are equivalent as observed in Observation~\ref{obs:immediate}.
    Since $\overline{K_3} \cong 3K_1$, $\overline{K_1 + K_2} \cong W_4$, 
    and $\overline{K_1 + C_5} \cong W_5$, the equivalence between the second and third
    statements is guaranteed by Theorem~\ref{thm:ageC3}. Thus, the last three 
    statements are equivalent, and clearly, the last item implies the first one. 
    Finally, the second item follows from the first one since 
    $W_4$ is not a Helly proper circular-arc graph (see, e.g., Corollary 5 
    in~\cite{linDAM116}), and $\overline{C_6}$ and $W_5$ are not even proper
    circular-arc graphs (see, e.g., Theorem 5 in~\cite{linDM309}), and thus
    not  unit circular-arc graphs.
\end{proof}

\section{Restating Brandt's condition}
\label{sec:brandt}
Given a graph $G$ and a vertex $v\in V(G)$, we denote by $N(v)$ the set of neighbours of $v$
(in $G$). Recall that two vertices $u,v\in V(G)$ are called \emph{twins} if $N(u) = N(v)$.  A
graph $G$ is called \textit{point-determining} if it
does not contain two distinct vertices that are twins.
Moreover, we say that $G$ is \textit{point-incomparable} if for any two distinct vertices
$u,v\in V(G)$ the sets $N(v)$ and $N(u)$ are incomparable. 

\begin{observation}\label{obs:twins}
If $\mathcal F$ is a set of point-determining graphs, then the class
of $\mathcal F$-free graphs is closed under addition of twins, i.e., if $u,v\in V(G)$ are distinct 
twins, then $G$ is $\mathcal F$-free if and only if $G-v$ is $\mathcal F$-free.
\end{observation}

\begin{lemma}\label{lem:extension->free}
    Consider a maximal triangle-free graph $G$. If $G$ does not contain $H_{10}-v$ as a subgraph,
    then $G$ is a $\{K_3,~K_1+2K_2,~K_1+C_5,~C_6\}$-free graph.
\end{lemma}
\begin{proof}
    Suppose for contradiction 
    that $G$ is maximal triangle-free but not 
    $\{K_3,~K_1+2K_2,~K_1+C_5,~C_6\}$-free. 
    First assume that $G$ contains $C_6$ as an induced subgraph; let
    $c_1,c_2,c_3,c_4,c_5,c_6$ be an induced $6$-cycle in $G$. Since $G$ is maximal
    triangle-free, there are three vertices $a_1,a_2,a_3$ such that $a_i$
    is a common neighbour of $c_i$ and $c_{i+3}$ for each $i\in \{1,2,3\}$.
    Since $G$ is triangle-free, the inequalities $a_1\neq a_2\neq a_3 \neq a_1$ hold. 
    Hence, this subgraph of $G$ induced by $c_1,\dots, c_6,a_1,a_2,a_3$ contains
    $H_{10}-v$ as a subgraph contradicting the choice of $G$. Therefore, 
    $G$ is $C_6$-free. Also notice that if there are three independent vertices
    $x_1,x_2,x_3$ that do not have a common neighbour, then we can find an 
    induced $C_6$ in $G$ by considering the common neighbour of $x_i$ and $x_j$
    for $1\le  i< j \le 3$. Thus, for here onward we assume that every $3$ independent
    vertices  have a common neighbour. 

    Suppose that $t$, $x$, $y$, $u$, and $v$ induce a copy of $K_1 + 2K_2$ in  $G$ 
    where $xy,uv\in E(G)$. Let $a$ be a common neighbour of $t,x,u$ and $b$
    a common neighbour of $t,y,u$. Since $G$ is triangle-free and $ta,tb\in E(G)$,
    the vertices $a$ and $b$ are not adjacent. From this observation and
    the fact that $x,y,u,v$ induce a copy of $2K_2$ in $G$, we conclude that 
    $a,x,y,b,v,u$ induce a $6$-cycle in $G$. In the paragraph above we proved that
    no such cycle exists in $G$. Therefore, $G$ is $(K_1 + 2K_2)$-free. 

    Finally, we assume that $G$ is not $(K_1 + C_5)$-free and conclude that $G$
    contains an induced copy of $K_1 + 2K_2$ contradicting our conclusion from 
    the previous paragraph. Suppose that $c_1,c_2,c_3,c_4,c_5$ is an induced
    $5$-cycle in $G$, and that $t$ is vertex of $G$ which is not adjacent to
    any vertex of the $5$-cycle. Let $a$ be a common neighbour of $t,c_1$ and
    $c_3$. Since $G$ is triangle-free, $a$ is not adjacent to any other vertex
    of the $5$-cycle. Therefore, the set  $\{c_2,t,a,c_4,c_5\}$ satisfy that
    $ta$ and $c_4c_5$ are the only edges with both end-vertices in this set,
    i.e., the set  $\{c_2,t,a,c_4,c_5\}$ induces a copy of $K_1 + 2K_2$ in $G$. 
    Putting these three paragraphs together, we conclude that 
    $G$ is $\{K_3, K_1 + 2K_2, K_1 + C_5, C_6\}$-free.
\end{proof}

The following two lemmas can be regarded as subcases of Lemma~\ref{lem:free->extension}, where we show that any 
$\{K_3,~K_1+2K_2,~K_1+C_5,~C_6\}$-free graph can be extended to 
a maximal triangle-free graph that is also
$\{K_1+2K_2,~K_1+C_5,~C_6\}$-free. Equivalently, every maximal
$\{K_3,~K_1+2K_2,~K_1+C_5,~C_6\}$-free graph is also maximal triangle-free.
In turn, we will use this to prove the main theorem of this section
where we propose equivalent conditions to Brandt's condition from 
Theorem~\ref{thm:brandt}.

To simplify our notation for adding edges to a graph $G$, given 
non-adjacent vertices $x,y \in V(G)$,  we denote by $G + xy$ the graph
$(V(G), E(G) \cup \{xy\})$. Notice that
we use the same notation as for the disjoint union of graphs, but it should
always be clear from the context to which operation we refer to. 

\begin{lemma}\label{lem:.--}
    Let $G$ be a $\{K_3,~K_1+2K_2,~K_1+C_5,~C_6\}$-free point-incomparable graph. 
    For any two distinct non-adjacent vertices $x,y \in V(G)$ either 
    \begin{itemize}
        \item $G + xy$ contains a triangle, or
        \item $G + xy$ is $\{K_3,~K_1+C_5,~C_6\}$-free.
    \end{itemize}
\end{lemma}
\begin{proof}
    Suppose that the first item does not hold, and anticipating a contradiction
    assume that $G + xy$ is not $\{K_3,~K_1+C_5,~C_6\}$-free. Suppose first that
    $G + xy$ contains an induced $5$-cycle $c_1,c_2,c_3,c_4,c_5$ and a vertex 
    $b$ not adjacent to any of $c_1,\dots, c_5$. Since $G$ is $K_1 + C_5$-free,
    it must be the case that, without loss of generality, $c_1 = x$ and $c_2 = y$. 
    Since $xy\not \in E(G)$, it follows that the set of vertices $\{b,c_5,x,y,c_2\}$
    induce a copy of $2K_2 + K_1$ in $G$, contradicting the choice of $G$. Now, 
    suppose that $G + xy$ contains an induced $6$-cycle. With similar arguments as before,
    we can assume that such cycle is of the form $c_1,c_2,c_3,c_4,c_5,c_6$ where
    $c_1 = x$ and $c_6 = y$. Since $G$ is point-incomparable, $c_6$ has a neighbour $a$
    (in $G$) a which is not a neighbour of $c_4$. Since $G$ is triangle-free,
    $a$ is not a neighbour of $c_5$, and since $G + c_1c_6$ does not contain
     a triangle, the vertex $a$ is not adjacent to $c_1$. The following is a depiction of
     the previous structure (in $G$) where dashed line segments represent non-edges
     (in $G$). 
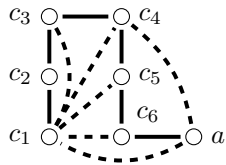
\begin{center}
\begin{tikzpicture}[scale = 0.8]

  \begin{scope}[xshift=5cm]
    \node [vertex, label = left:{$c_1$}] (1) at (-0.6,0) {};
    \node [vertex, label = left:{$c_2$}] (2) at (-0.6,1) {};
    \node [vertex, label = left:{$c_3$}] (3) at (-0.6,2) {};
    \node [vertex, label = right:{$c_4$}] (4) at (0.6,2) {};
    \node [vertex, label = right:{$c_5$}] (5) at (0.6,1) {};
    \node [vertex, label = 45:{$c_6$}]  (6) at (0.6,0) {};

    \node [vertex, label = right:{$a$}]  (a) at (1.8,0) {};

    \foreach \from/\to in {1/2, 2/3, 3/4, 5/6, 4/5, 6/a}     
    \draw [edge] (\from) to (\to);

    \foreach \from/\to in {6/1, 1/4, 1/5}     
    \draw [edge, dashed] (\from) to (\to);
    \draw [edge, dashed] (a) to [bend right = 20] (4);
    \draw [edge, dashed] (3) to [bend left] (1);
    \draw [edge, dashed] (a) to [bend left] (1);
  \end{scope}

\end{tikzpicture}
\end{center}
     Finally, we consider two possible cases, either $c_3a\in E$ or
     $c_3a\not\in E$. In the former, the vertices $c_3,c_4,c_5,c_6,a$ induce
     a $5$-cycle in $G$, and $c_1$ is not adjacent to neither of these,
    contradicting the fact that $G$ is $K_1 + c_5$-free. Otherwise, the set
    $\{c_1,c_3,c_4,c_6,a\}$ induces a copy of $K_1 + 2K_2$ in $G$, contradicting
    again the choice of $G$.
    Therefore, if $G + xy$ does not contain a triangle, then $G + xy$ is  a 
    $\{K_3,~K_1+C_5,~C_6\}$-free graph. 
\end{proof}

Building on this lemma we prove the following one. 

\begin{lemma}\label{lem:C6}
    Let $G$ be a $\{K_3,~K_1+2K_2,~K_1+C_5,~C_6\}$-free point-incomparable graph. 
    For any two distinct non-adjacent vertices $x,y \in V(G)$  either 
    \begin{itemize}
        \item $G + xy$ contains a triangle, or 
        \item $G + xy$ is $\{K_3,K_1+2K_2\}$-free.
    \end{itemize}
\end{lemma}
\begin{proof}
Our proof is by contradiction. Suppose that $t, x,y, u, v$ induce a copy of $K_1+ 2K_2$ in $G+xy$,
and that $G+xy$ is triangle-free, i.e, $x$ and $y$ do not have a common neighbour in $G$. 
Since $G$ is point-incomparable, there is a vertex $a$ adjacent to $y$ but not to $t$.
Observe that $a$ must be a neighbour of $u$ or $v$ because otherwise $t$, $y$, $a$, 
$u$, $v$ would induce a copy of $K_1 + 2K_2$ in $G$. Also $u,v,a$ cannot induce
 a triangle in $G$, so we assume without loss of generality that $a$ is adjacent
 to $v$ and not to $u$. Again, using the fact that $G$ is point-indistinguishable, 
 we find a vertex   $b$ adjacent to $u$ but not to  $a$. The following diagram
 depicts the current scenario, where dashed edges represent non-adjacent vertices
 in $G$. 
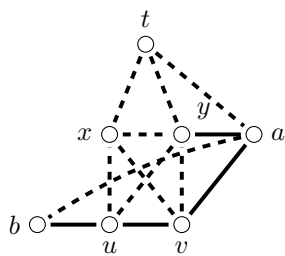
\begin{center}
    \begin{tikzpicture}[scale = 0.8]

    \begin{scope}[xshift=5cm]
        \node [vertex, label = above:{$t$}] (t) at (0,3) {};
        \node [vertex, label = left:{$x$}] (x) at (-0.6,1.5) {};
        \node [vertex, label = 45:{$y$}] (y) at (0.6,1.5) {};
        \node [vertex, label = below:{$u$}] (u) at (-0.6,0) {};
        \node [vertex, label = below:{$v$}] (v) at (0.6,0) {};
    
        \node [vertex, label = right:{$a$}]  (a) at (1.8,1.5) {};
        \node [vertex, label = left:{$b$}]  (b) at (-1.8,0) {};

        \foreach \from/\to in {u/v, a/y, b/u, v/a} 
        \draw [edge] (\from) to (\to);

        \foreach \from/\to in {t/x, t/y, t/a, u/x, u/y, v/x, v/y, x/y} 
        \draw [edge, dashed] (\from) to (\to);
        \draw [edge, dashed] (a) to [bend right = 10] (b);
    \end{scope}
    \end{tikzpicture}
\end{center}
Since $G + xy$ does not create a triangle, $b$ is a neighbour of at most
one of $x$ or $y$. We conclude by considering the remains three possible cases:
\begin{itemize}
    \item $b$ is not a neighbour of $x$ nor $y$. So, $x,b,u,y,a$ create
    a copy of $K_1 + 2K_2 $ in $G$, contradicting the choice of $G$. 
    \item $b$ is adjacent to $y$ but not to $x$. So, $x$ together
    with $b,y,a,u,v$ create a copy of $K_1 + C_5$ in $G$, contradicting again the
    choice of $G$. 
    \item $G$ is adjacent to $x$ but not to $y$. In this case, 
    $x,b,u,v,a,y$ induce a $6$-cycle in $G+xy$, but this, together with 
    the fact that $G+xy$ is triangle-free, contradicts Lemma~\ref{lem:C6} .
\end{itemize}
Therefore, we conclude that either $G + xy$ is not triangle-free, or 
$G+x y$ is $\{K_3,K_1 + 2K_2\}$-free.
\end{proof}

\begin{lemma}\label{lem:free->extension}
    Every $\{K_3,~K_1+2K_2,~K_1+C_5,~C_6\}$-free graph $G$ can be extended to a maximal 
    triangle-free graph that contains neither $K_1+2K_2$, $~K_1+C_5$, nor $C_6$ as an induced 
    subgraph.
\end{lemma}
\begin{proof}
    First, suppose that $G$ is not point-indistinguishable, i.e., there are 
    vertices $x,y\in V(G)$ such that $N(x)\subseteq N(y)$. In this case,
    we can remove $x$ from $G$, and from any $\{K_3,~K_1+2K_2,~K_1+C_5,~C_6\}$-free
    extension $(G-x)'$ of $G$, we construct a $\{K_3,~K_1+2K_2,~K_1+C_5,~C_6\}$-free
    extension of $G$ by adding $x$ as twin of $y$ with respect to $(G-x)'$. To be
    precise, $V(G') = V(G)$ and $E(G') = E((G-x)') \cup \{xu\colon   yu \in E((G-x)')\} $.
    It is straightforward to observe that if $(G-x)'$ is maximal triangle-free,
    then $G'$ maximal triangle-free, and it follows from Observation~\ref{obs:twins}
    that if $(G-x)'$ is $\{K_1+2K_2,~K_1+C_5,~C_6\}$-free, then  $G'$ is $\{K_1+2K_2,~K_1+C_5,~C_6\}$-free too. 

    Now suppose that $G$ is point-indistinguishable. If $G$ is not maximal
    $\{K_3,~K_1+2K_2,~K_1+C_5,~C_6\}$-free, then we add any edge $xy$ such that
    $G + xy$ is a $\{K_3,~K_1+2K_2,~K_1+C_5,~C_6\}$-free graph. Otherwise, 
    $G$ is maximal $\{K_3,~K_1+2K_2,~K_1+C_5,~C_6\}$-free, hence, by
    Lemmas~\ref{lem:.--} and~\ref{lem:C6}, we conclude that $G$ is maximal
    triangle-free. 

    Summarizing, given a $\{K_3,~K_1+2K_2,~K_1+C_5,~C_6\}$-free graph $G$
    we add edges until $G$ is not point-indistinguishable, or $G$ is 
    a maximal $\{K_3,~K_1+2K_2,~K_1+C_5,~C_6\}$-free graph. In the former case,
    we remove a vertex $u$ such that $N(u) \subseteq N(v)$ for some $v\neq u$,
    and inductively find the desired extension of $G$; in the latter case
    $G$ is a maximal triangle-free graph  that does not contain $K_1+2K_2,~K_1+C_5$,
    nor $C_6$ as induced  subgraphs.
\end{proof}

Building on the lemmas proved on this section, we extend Corollary~\ref{cor:circular<3} to 
include, and independently reprove, Brandt's original characterization of $\chi_c(G) < 3$
(Theorem~\ref{thm:brandt}).

\begin{theorem}\label{thm:equivalent-extensions}
    The following statements are equivalent for any finite graph $G$.
\begin{enumerate}
        \item $\chi_c(G) < 3$.
        \item $G\in \Csp(\mathbb C_3)$.
        \item $G$ can be extended to a $\{K_3,~K_1+2K_2,~K_1+C_5,~C_6\}$-free graph. 
        \item $G$ can be extended to a maximal triangle-free graph that
        is $\{K_1+2K_2,~K_1+C_5,~C_6\}$-free.
        \item $G$ can be extended to a maximal triangle-free graph that does not
        contain $H_{10}-v$ as a subgraph.
        \item $\overline{G}$ contains a spanning $3K_1$-free unit Helly circular-arc subgraph. 
\end{enumerate}
\end{theorem}
\begin{proof}
    The equivalence between the first three statements holds as stated in Corollary~\ref{cor:circular<3}. 
    The third statement implies the fourth one via Lemma~\ref{lem:free->extension}.
    Clearly, if $G$ is a triangle-free graph and contains $H_{10}-v$ as a subgraph, then
    $G$ contains $C_6$ as an induced subgraph. Thus, if $G$ can be extended to a maximal
    triangle-free graph $G'$ that contains neither $K_1+2K_2$, $K_1+C_5$, nor $C_6$ as an induced
    subgraph, then $G'$ is a maximal triangle-free extension of $G$ that does not contain
    $H_{10}-v$ as an induced subgraph. 
    So the fourth statement implies the firth one. 
    The fifth one implies the third one via Lemma~\ref{lem:extension->free}, and so
    statements 1 -- 5 are equivalent. Finally, the equivalence between the third and
    sixth statement follows
    from Theorem~\ref{thm:UCA} by taking complements and noticing that 
    $\overline{K_3} \cong 3K_1$, $\overline{K_1+2K_2} \cong W_4$, and
    $\overline{K_1+C_5} \cong W_5$. 
\end{proof}

We conclude this section with a brief discussion about a natural generalization
of some of the equivalences in Theorem~\ref{thm:equivalent-extensions}, addressing Zhu's
comment regarding generalizations of Brandt's characterization of $\chi(G) < 3$ (Theorem~\ref{thm:brandt})
to larger integers $k$~\cite{zhuDM229}.

For a positive integer $k$ we say that a graph $G$ is a \textit{circular $K_k$-free graph}
if every vertex $x$ of $G$ can be represented by a point $p_x$ of a circle in such a way that
$xy \in E(G)$ if and only if $p_x$ and $p_y$ lie at an angle larger that $2\pi/k$. Similarly as in
the case $k = 3$, i.e., circular triangle-free graphs, every circular $K_k$-free graph is $K_k$-free.
It follows from the definition of the circular chromatic number via circular colourings (see,
e.g.,~\cite[Definition 1.1]{zhuDM229})  that a graph $G$ satisfies $\chi_c(G) < k$ if and only if it can be
extended to a circular $K_k$-free graph. Thus, if $\mathcal F_k$ is the set of minimal obstructions of
circular $K_k$-free graphs, then $\chi_c(G) < k$ if and only if $G$ can be extended to an $\mathcal F_k$-free
graph.

\begin{problem}
    Determine the minimal obstructions of circular $K_k$-free graphs. 
\end{problem}

\section{Axiomatization}
\label{sec:expansion}

In this section we propose a finite axiomatization of the first-order theory of $\mathbb C_3$. For this, it will
be convenient to expand our signature by a relation symbol $B$ of arity three and a relation symbol $S$ of arity four 
which are both existentially and universally definable in $\mathbb C_3$. Thus, any axiomatization
of $(\mathbb C_3, B, S)$ has a syntactic translation to an axiomatization of $\mathbb C_3$
(using only the signature of graphs). We will also show that $(\mathbb C_3, B, S)$ is a
homogeneous structure, and from this, we will see that Theorem~\ref{thm:ageC3} extends to countable graphs.
The fact that ${\mathbb C}_3$ has a universal and existential homogeneous expansion will also be used in
Section~\ref{sec:complexity}.

\subsection{Local betweenness and total separation}
To begin with, we define the ternary relation $B$ on $V(\mathbb C_3)$ using the
geometric construction of $\mathbb C_3$: 
for elements $x,y,z \in V(\mathbb C_3)$,  the relation $B(x,y,z)$ holds if $\{x,y,z\}$ is an
independent set in ${\mathbb C}_3$, and $y$ lies \textit{between} $x$ and $z$, i.e., $y$ lies
in the circular arc defined by the acute angle between $x$ and $z$. It is straightforward
to observe that $B(x,y,z)$ has the following universal and existential definitions in $\mathbb C_3$.

\begin{observation}\label{obs:B-definitions}
    The following  are equivalent for any three mutually non-adjacent vertices $x,y,z\in V(\mathbb C_3)$.
    \begin{itemize}
        \item $\mathbb C_3\models B(x,y,z)$. 
        \item $\mathbb C_3\models \forall v \big(E(y,v) \implies E(x,v) \lor E(z,v)\big)$.
        \item $\mathbb C_3\models \exists u, v \big(E(x,v) \land \lnot E(y,v) \land \lnot E(z,v)
        \land E(z,u) \land \lnot E(y,u) \land \lnot E(x,u)\big)$.
    \end{itemize}
\end{observation}

Now, we consider the standard \textit{separation} relation $S$ on $V(\mathbb C_3)\subseteq S_1$:
for $x_1,\dots,x_4 \in V({\mathbb C}_3)$ the relation 
$S(x_1,x_2,x_3,x_4)$ holds if and only if $x_2$ and $x_4$ lie in different (topologically) connected
components of $S_1\setminus\{x_1,x_3\}$; equivalently, the relation holds if when we traverse $S_1$ in a clockwise motion
starting in $x_1$ we see $x_2$, $x_3$, and then $x_4$, or we see $x_4$, $x_3$, and then $x_2$.

Notice that $S$ is a \textit{total} relation on $V(\mathbb C_3)$, while $B$ is not, i.e., $B$ is only defined on triples that induce an independent set (of three vertices),
while $S$ is defined on any four-tuple of (pair-wise distinct) vertices. 
For this reason, the proposed universal and existential definitions of $B(x_1,x_2,x_3)$  have no case
distinction on the graph induced by $x_1,x_2,x_3$. In contrast, the universal and existential
definitions of $S$ proposed below  depend on the graph induced by  $x_1,x_2,x_3,x_4$.
If Figure~\ref{fig:S-cases} we depict all possible configurations up to automorphisms of $\mathbb C_3$.


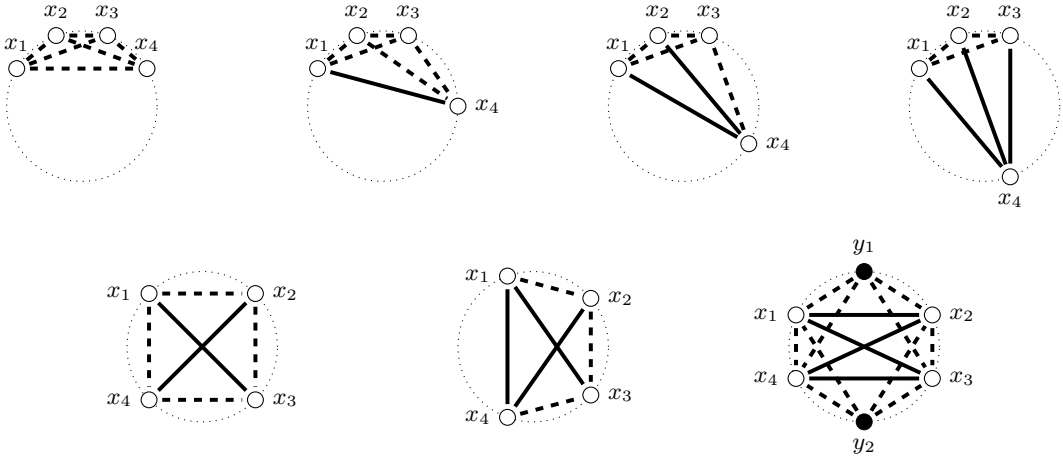
\begin{figure}[ht!]
\centering
\begin{tikzpicture}[scale = 0.8]


\begin{scope}[scale=0.5] 
\draw [black, dotted] (0,0) circle[radius = 2.5];
\node [vertex, label = above:{\small $x_1$}] (1) at (150:2.5){};
\node [vertex, label =above:{\small $x_2$}] (2) at (110:2.5){};
\node [vertex, label = above:{\small $x_3$}] (3) at (70:2.5){};
\node [vertex, label = above:{\small $x_4$}] (4) at (30:2.5){};

\foreach \from/\to in {1/2, 1/3, 1/4, 2/3, 2/4, 3/4}
\draw [edge, dashed] (\from) to (\to);

\end{scope}

\begin{scope}[xshift = 5cm,  scale=0.5] 
\draw [black, dotted] (0,0) circle[radius = 2.5];
\node [vertex, label = above:{\small $x_1$}] (1) at (150:2.5){};
\node [vertex, label =above:{\small $x_2$}] (2) at (110:2.5){};
\node [vertex, label = above:{\small $x_3$}] (3) at (70:2.5){};
\node [vertex, label = right:{\small $x_4$}] (4) at (0:2.5){};

\draw [edge] (1) to (4);
\foreach \from/\to in {1/2, 2/3, 1/3, 2/4, 3/4}
\draw [edge, dashed] (\from) to (\to);

\end{scope}

\begin{scope}[xshift = 10cm,  scale=0.5] 
\draw [black, dotted] (0,0) circle[radius = 2.5];
\node [vertex, label = above:{\small $x_1$}] (1) at (150:2.5){};
\node [vertex, label =above:{\small $x_2$}] (2) at (110:2.5){};
\node [vertex, label = above:{\small $x_3$}] (3) at (70:2.5){};
\node [vertex, label = right:{\small $x_4$}] (4) at (-30:2.5){};

\draw [edge] (2) to (4);
\draw [edge] (1) to (4);
\foreach \from/\to in {1/2, 2/3, 1/3, 3/4}
\draw [edge, dashed] (\from) to (\to);

\end{scope}

\begin{scope}[xshift = 15cm,  scale=0.5] 
\draw [black, dotted] (0,0) circle[radius = 2.5];
\node [vertex, label = above:{\small $x_1$}] (1) at (150:2.5){};
\node [vertex, label =above:{\small $x_2$}] (2) at (110:2.5){};
\node [vertex, label = above:{\small $x_3$}] (3) at (70:2.5){};
\node [vertex, label = below:{\small $x_4$}] (4) at (290:2.5){};

\foreach \from/\to in {4/1, 4/2, 4/3}
\draw [edge] (\from) to (\to);
\foreach \from/\to in {1/2, 1/3, 2/3}
\draw [edge, dashed] (\from) to (\to);

\end{scope}


\begin{scope}[yshift = -4cm, xshift = 2cm, scale=0.5] 
\draw [black, dotted] (0,0) circle[radius = 2.5];
\node [vertex, label = left:{\small $x_1$}] (1) at (135:2.5){};
\node [vertex, label =right:{\small $x_2$}] (2) at (45:2.5){};
\node [vertex, label = right:{\small $x_3$}] (3) at (-45:2.5){};
\node [vertex, label = left:{\small $x_4$}] (4) at (225:2.5){};

\draw [edge] (1) to (3);
\draw [edge] (2) to (4);
\foreach \from/\to in {1/2, 2/3, 4/3, 1/4}
\draw [edge, dashed] (\from) to (\to);

\end{scope}

\begin{scope}[yshift = -4cm, xshift = 7.5cm,  scale=0.5] 
\draw [black, dotted] (0,0) circle [radius = 2.5];
\node [vertex, label = left:{\small $x_1$}] (1) at (110:2.5){};
\node [vertex, label =right:{\small $x_2$}] (2) at (40:2.5){};
\node [vertex, label = right:{\small $x_3$}] (3) at (-40:2.5){};
\node [vertex, label = left:{\small $x_4$}] (4) at (250:2.5){};

\foreach \from/\to in {1/4, 1/3, 4/2}
\draw [edge] (\from) to (\to);
\foreach \from/\to in {1/2, 2/3, 4/3}
\draw [edge, dashed] (\from) to (\to);

\end{scope}

\begin{scope}[yshift = -4cm, xshift = 13cm,  scale=0.5] 
\draw [black, dotted] (0,0) circle[radius = 2.5];
\node [vertex, label = left:{\small $x_1$}] (1) at (155:2.5){};
\node [vertex, label =right:{\small $x_2$}] (2) at (25:2.5){};
\node [vertex, label = right:{\small $x_3$}] (3) at (-25:2.5){};
\node [vertex, label = left:{\small $x_4$}] (4) at (205:2.5){};
\node [vertex, fill = black, label = above:{\small $y_1$}] (y1) at (90:2.5){};
\node [vertex, fill = black, label = below:{\small $y_2$}] (y2) at (270:2.5){};

\foreach \from/\to in {1/2, 2/4, 1/3, 3/4}
\draw [edge] (\from) to (\to);
\foreach \from/\to in {1/4, 2/3, y1/1, y1/2, y1/3, y1/4, y2/1, y2/2, y2/3, y2/4}
\draw [edge, dashed] (\from) to (\to);

\end{scope}
\end{tikzpicture}

\caption{Depiction of the possible scenarios of the separation relation on $\mathbb C_3$, up to cyclic permutations
and reflections of $(x_1,x_2,x_3,x_4)$, i.e., compositions of $(x_1,x_2,x_3,x_4)\mapsto (x_2,x_3,x_4,x_1)$ 
and $(x_1,x_2,x_3,x_4)\mapsto (x_4,x_3,x_2,x_1)$. The vertices $y_1$ and $y_2$ in the bottom right picture represent the
existentially quantifies variable from the third item in Observation~\ref{obs:S-definitions}.}
\label{fig:S-cases}
\end{figure}

It is not hard to notice that $S$ has the cyclic symmetry  $S(x_1,x_2,x_3,x_4) \Leftrightarrow
S(x_2,x_3,x_4,x_1)$, and the reflection symmetry $S(x_1,x_2,x_3,x_4) \Leftrightarrow S(x_4,x_3,x_2,x_1)$.
In order to shorten our writing in the proposed existential and universal definitions of $S$, it will be
convenient to define $S$ up to cyclic permutations and reflections of $(x_1,x_2,x_3,x_4)$, i.e.,
composition of $(x_1,x_2,x_3,x_4)\mapsto (x_2,x_3,x_4,x_1)$  and $(x_1,x_2,x_3,x_4)\mapsto (x_4,x_3,x_2,x_1)$.
We will also group all configurations depicted in Figure~\ref{fig:S-cases} into three cases: $\{x_1,x_2,x_3,x_4\}$
contains an independent subset of three vertices; the tuple induces a $P_4$ or a $2K_2$; and the tuple induces a
$4$-cycle. These are indeed all possible configurations as any other graph on four vertices has a triangle (and
$\mathbb C_3$ is triangle-free).

\begin{observation}\label{obs:S-definitions}
    Two vertices $x_1,x_3\in V(\mathbb C_3)$ separate  $x_2,x_4\in V(\mathbb C_3)$,
    i.e., $S(x_1,x_2,x_3,x_4)$, if and only if one of the following holds (up to cyclic permutations
    and reflections of $(x_1,x_2,x_3,x_4)$):
    \begin{itemize}
        \item $B(x_1,x_2,x_3)$ and $\lnot B(x_1, x_4, x_3)$ (at least three vertices are mutually non-adjacent).
        \item $E(x_1,x_3)\land E(x_2, x_4)$ and $\lnot E(x_1,x_2)\land\lnot E(x_2,x_3)\land \lnot E(x_3,x_4)$
        (induced $P_4$ or $2K_2$).
        \item $E(x_1,x_3)\land E(x_3,x_4)\land E(x_4,x_2)\land E(x_2,x_1)$ (the tuple induces
        a $4$-cycle), and there are vertices $y_1,y_2$ such that
        $B(x_4,x_1,y_1)\land B(y_1,x_2,x_3)$ and $B(y_2,x_4,x_1)\land B(x_2,x_3,y_2)$, equivalently,
        \item $E(x_1,x_3)\land E(x_3,x_4)\land E(x_4,x_2)\land E(x_2,x_1)$ (the tuple induces
        a $4$-cycle), and for every $y$ such that $\lnot E(y,x_i)$
        with $i\in\{1,2,3,4\}$, then $B(x_4,x_1,y)\land B(y,x_2,x_3)$ or $B(y,x_4,x_1)\land B(x_2,x_3,y)$.
    \end{itemize}
\end{observation}

\subsection{Universal axioms and implications}

We first consider a sequence of universal axioms satisfied by $(\mathbb C_3, B, S)$. We divide them
into edge axioms, betweenness axioms,  separation axioms, and mixed axioms. Whenever it is convenient,
we write the axioms in plain English. All variables are
implicitly universally quantified. \\

\noindent \textit{Universal edge axioms}

\vspace{0.1cm}
$\UEI$ $E$ is a loopless symmetric relation.

$\UEII$ $E$ contains no triangles.

\vspace{0.2cm}
\noindent \textit{Universal local betweenness axioms}

\vspace{0.1cm}
$\UBI$ $B(x,y,z)$ if and only if $B(z,y,x)$.

$\UBII$ If $B(x,y,z)$, then $\lnot B(y,z,x)$.

$\UBIII$ If $B(x,y,z)$ and $B(y,w,z)$, then $B(x,w,z)$ and $B(x,y,w)$.

$\UBIV$ If $B(x,y,z)$ and $B(x,y,w)$, then $B(x,z,w)$ or $B(x,w,z)$.

\vspace{0.2cm}
\noindent \textit{Universal separation axioms}

\vspace{0.1cm}
$\USI$ Exactly one of $S(x_1,x_2,x_3,x_4)$, $S(x_1,x_2,x_4,x_3)$, or $S(x_1,x_3,x_2,x_4)$ holds.

$\USII$ $S(x_1,x_2,x_3,x_4)$ if and only if $S(x_2,x_3,x_4,x_1)$.

$\USIII$ $S(x_1,x_2,x_3,x_4)$ if and only if $S(x_4,x_3,x_2,x_1)$.

$\USIV$ If $S(x_1,x_2,x_3,x_4)$ and $S(x_1,x_3,x_4,x_5)$, then $S(x_1,x_2,x_4,x_5)$.

\vspace{0.2cm}
\noindent \textit{Universal mixed axioms}

\vspace{0.1cm}
$\UMI$ $\{x,y,z\}$ is an independent set of size 3 if and only if $B(x,y,z)$, $B(y,z,x)$, 

\phantom{$\UMI$} or $B(z,x,y)$.

$\UMII$ If $B(x,y,z)$ and $B(y,z,w)$ and $\lnot E(x,w)$, then $B(x,y,w)$ and $B(x,z,w)$.

$\UMIII$ If $B(x,y,z)$ and $E(y,w)$, then $E(x,w)$ or $E(z,w)$.

$\UMIV$ If $B(x,y,z)$,  $E(x,w)$, and $E(z,w)$, then $E(y,w)$.

$\UMV$ If $S(x_1,x_2,x_3,x_4)$, $E(x_1,x_4)$, and $B(x_3,x_5,x_4)$, then $S(x_1,x_2,x_3,x_5)$.

$\UMVI$ If $B(y_1,y_2,x)$, $B(x,z_1,z_2)$, and $E(y_1,z_2)$, then $S(y_1,y_2,z_1,z_2)$.

$\UMVII$ If $x_1,x_2,x_3,x_4$ induce a $C_4$, a $P_4$ or a $2K_2$, then  $S(x_1,x_2,x_3,x_4)\Rightarrow E(x_1,x_3)\land E(x_2, x_4)$.

$\UMVIII$ If $\{x_1,x_2,x_3\}$ is an independent set of size 3  and $x_4 \notin \{x_1,x_2,x_3\}$, then

\phantom{$\UMVIII$} $S(x_1,x_2,x_3,x_4)$ if and only if $B(x_1,x_2,x_3) \land \lnot B(x_1,x_4,x_3)$ or 
$B(x_1,x_4,x_3)\land \lnot B(x_1,x_2, x_3)$.

\vspace{0.25cm}

\begin{observation}\label{obs:ax}
It follows from the geometric construction that $(\mathbb C_3, B,S)$ satisfies the universal edge axioms, 
local betweenness axioms, and separation axioms. 
Also the mixed axioms $\UMI$, $\UMII$,  $\UMIV$, and $\UMV$ follow from the geometric
construction of $(\mathbb C_3, B,S)$, while the rest of the mixed axioms follow from Observations~\ref{obs:B-definitions}
and~\ref{obs:S-definitions} --- each of these observations are also consequences of the geometric
construction of $(\mathbb C_3, B, S)$.
\end{observation} 

Any relation $S$ of arity four  satisfying the universal separation axioms listed above is called a
\textit{separation} relation. Given a separation relation $S$ on a set $V$, we say that 
$u,v\in V$ \textit{separate} $u',v'$ if $S(u,u',v,v')$.
Recall that $V(\mathbb C_3)$ is a generic dense
countable subset of $S_1$. 
Any countable $\{S\}$-structure $(V,S)$ such that $S$ is a separation
relation on $V$ embeds into $(V(\mathbb C_3), S)$, and moreover $(V(\mathbb C_3), S)$ is a homogeneous
structure (see, e.g., \cite{Oligo}).


A ternary relation $B$ satisfying the universal betweenness axioms and totality
(for any three distinct vertices $x_1,x_2,x_3$ either $B(x_1,x_2,x_3)$, $B(x_2,x_3,x_1)$ or
$B(x_3,x_2,x_1)$) is called a \textit{betweenness} relation. In particular, by axiom $\UMI$, 
for any $\{E,B,S\}$-structure $\bA$ with $E(\bA)$ empty, the interpretation of $B$ in $\bA$ defines a
betweenness relation.

Betweenness relations arise from considering the ternary relation
$$\Bet := \{(x,y,z) \in {\mathbb Q}^3 \mid x<y<z \text{ or } z < y < x\}$$ 
where $<$ is the natural
strict linear order on ${\mathbb Q}$. As one might expect,  $(\mathbb Q, \Bet)$ is the universal
homogeneous countable betweenness relation, up to isomorphism (see, e.g.,~\cite{Oligo}).
Clearly, if $U\subseteq V(\mathbb C_3)$ is a circular arc of length strictly less than $1/3$,
then $U$ is an independent set (with respect to $E$) and so, $(U,B)$ is a countable betweenness
relation, and thus, it embeds into $(\mathbb Q, \Bet)$. Actually, it is straightforward to observe that
$(U,B) \cong (\mathbb Q, \Bet)$.

\begin{lemma}\label{lem:unique-bet}
Let $V$ be a countable set with a betweenness relation $B$, and let $u,w\in V$ be distinct 
with no element in $V$ between $u$ and $w$. For any new element $v\not\in V$,
there is a unique betweenness relation $B'$ on $V\cup \{v\}$ such that $B\subseteq B'$,
and $B'(u,v,w)$. Similarly, if $x,y\in V$ are distinct and every element of $V\setminus\{x,y\}$ is
between $x$ and $y$, then there is a unique betweenness relation $B'$ on $V\cup \{v\}$ such that
$B\subseteq B'$  and $B'(v,x,y)$.
\end{lemma}
\begin{proof}
    This lemma can be proved by considering the unique embedding of $(V,B)$ into $(\mathbb Q,\Bet)$, up
    to isomorphism, and then appropriately identifying $v$ with some element of $\mathbb Q$. 
\end{proof}

As seen above, the universal separation relation and local betweenness axioms already have some
structural implications regarding $\{E,B,S\}$-structures satisfying these axioms.  In particular, 
these relate to well-known and previously studied relations. Now, we observe that the
universal mixed axioms force some strong interplay between the relations $E$, $B$, and $S$.
To begin with, we highlight the following immediate implication of axioms the $\UMI-\UMVIII$.

\begin{observation}\label{obs:sep-indeptendent}
Let $\bA$ be a finite $\{E,B,S\}$-structure satisfying the edge, betweenness, separation and mixed universal axioms,
and $a\in V(\bA)$ such that $V(\bA)\setminus\{a\}$ is an independent set (with respect to $E$). If $S'$ is a separation
relation on $V(\bA)$ such that $(V(\bA),B,S')$ satisfies the universal axioms, then $S' = S$. 
\end{observation}
\begin{proof}
    For any set $\{x_1,x_2,x_3,x_4\}\subseteq V$ of four vertices there is an independent subset
    of three vertices, and so axiom $\UMI$ implies that $B(x_i,x_j,x_k)$ for some $i,j,k\in \{1,\dots, 4\}$. 
    We then conclude by axiom $\UMVIII$ (and the universal separation and local betweenness axioms) that
    $S(x_1,x_2,x_3,x_4)$ if and only if
    $S'(x_1,x_2,x_3,x_4)$.
\end{proof}

Given a subset of vertices $U$ in  an $\{E,B,S\}$-structure $\mathbb A$, we write $B(U)$ to denote the set
$$U\cup \{v\in V(\mathbb A)\colon B(u_1, v, u_2),  \text{ for some } u_1,u_2\in U\}.$$

\begin{lemma}\label{lem:order}
    Let $\bA$ be a countable $\{E,B,S\}$-structure  satisfying the universal axioms $\UBI$, $\UBII$, $\UBIII$, $\UMI$,
    $\UMII$, $\UMIII$, and $\UMIV$, let $a,b \in V(\bA)$ be non-adjacent, and let 
    $U = B(\{a,b\})$. Then $U$ is an independent set, and
    $$<_{ab} \; := \big \{(c,d) \in S \mid B(a,c,d) \wedge B(c,d,b) \big \} \cup \big \{(a,c) \mid c \in U\setminus \{a\} \big \} \cup \big \{(c,b) \mid c \in U\setminus \{b\} \big \}$$
    defines a (strict) linear order on $U$.
\end{lemma}
\begin{proof}
    It follows from $\UMI$ that the restriction of $B$ to any independent set is a betweenness relation, and thus
    it embeds into $(\mathbb Q, \Bet)$. So if $B(\{a,b\})$ is an independent set, then it follows that 
    $<_{ab}$ is a linear order of $B(\{a,b\})$ by considering the unique embedding of $(B(\{a,b\},B)$ into
    $(\mathbb Q,\Bet)$ (up to isomorphisms of $(\mathbb Q,\Bet)$. We now prove that $B(\{a,b\})$ is an independent
    set. By axiom  $\UMI$, and $\lnot E(a,b)$, it follows that $a$ and $b$ are not adjacent to any vertex in
    $B(\{a,b\})$. Now, let $c$ be such that $B(a,c,b)$ and let $d$ be a neighbour of $c$. By axiom $\UMIII$, $c$
    is a neighbour of $a$ or $b$, and thus, $\lnot B(a,d,b)$ (by axiom $\UMI$ again). Therefore, if $c,d \in V$ are such that $B(a,c,b)$ and $B(a,d,b)$,
    then they cannot be adjacent, so $B(\{a,b\})$ is an independent set.
\end{proof}

\begin{corollary}\label{cor:order-indep}
    Let $\bA$ be an $\{E,B,S\}$-structure  satisfying the universal axioms $\UBI$, $\UBII$, $\UBIII$, $\UMI$,
    $\UMII$, $\UMIII$, and $\UMIV$.
    Then every finite non-empty $U \subseteq V$ which forms an independent set with respect to $E$ 
    contains a unique pair $\{s,t\} \subseteq U$  
    such that $U \subseteq B(\{s,t\})$ and $<_{st}$ and $<_{ts}$ define a linear order on $U$.
\end{corollary}

\begin{proof}
    The statement follows by considering  $(U,B)$ as an induced substructure of $(\mathbb Q, \Bet)$,
    and choosing $s$ and $t$, the minimum and maximum of $U$ with respect to the linear order of $\mathbb Q$.
    The fact that $<_{st}$ defined a linear order on $U$ follows from Lemma~\ref{lem:order}, and
    the uniqueness of $\{s,t\}$ follows from the fact that there is a unique embedding of $(U,B)$ into
    $(\mathbb Q, \Bet)$, up to isomorphisms of $(\mathbb Q,\Bet)$.
\end{proof}

Given a finite independent set $U$ (or a finite $\{E,B,S\}$-structure $\bA$ with empty interpretation of $E$), 
we call the pair $\{s,t\}$ given in Corollary~\ref{cor:order-indep} the \textit{pair of $B$-bounds} of
$U$ (resp.\ of $\bA$); note that $s$ might be equal to $t$.

\begin{lemma}\label{lem:V-independent}
    Let $\bA$ and $\bB$ be $\{E,B,S\}$-structures satisfying the universal axioms.  Let $\bA'$ be a finite
    substructure of $\bA$ and $v\in V(\bA)\setminus V(\bA')$. If $V(\bA')$ is an independent set (with respect
    to $E$) and $u,w\in V(\bA')$ are such that $\{u,w\}$ is a pair of $B$-bounds of $\bA'$,
    then (exactly) one of the following holds for some $a,b\in V(\bA')$ with $a <_{uw} b$ and such that
    $\lnot B(a,c,b)$ for every $c\in V(\bA')$.
        \begin{enumerate}
            \item $E(u,v)$ and $E(w,v)$.
            \item $B(v,u,w)$, $B(u,w,v)$, or $B(a,v,b)$.
            \item $B(b,w,v)$, $E(a,v)$, and $E(u,v)$.
            \item $B(v,u,a)$, $E(b,v)$, and $E(w,v)$.
            \end{enumerate}
    Moreover,  in each of the cases above, if $f\colon \bA'\to \bB$ is an embedding and there is
    $v'\in V(\bB)\setminus f[V(\bA')]$ such that $\bB$ models the same atomic formulas substituting
    $v$ by $v'$ (and  $u,w,a,b$ for $f(u), f(w), f(a), f(b)$, respectively), then
    the extension of $f$  mapping $v$ to $v'$ is an embedding of the substructure of $\bA$ with vertex set
    $V(\bA')\cup\{v\}$ into $\bB$.
\end{lemma}
\begin{proof}
    Suppose that $v$ is not adjacent to both $u$ and $w$. First consider the case when it is not adjacent
    to neither of the two. In this case, axiom $\UMIII$ implies that $v$ has no neighbour in $V(\bA')$.  Hence,
    using the linear order $<_{uw}$ and Corollary~\ref{cor:order-indep}, we conclude that if $\lnot B(v,u,w)$
    and $\lnot B(u,w,v)$, then there are such $a,b$ with $B(a,v,b)$.
    Secondly, suppose that $E(u,v)$ and $\lnot E(w,v)$ and let $a \in  V(\bA')$ be the $<_{uw}$-minimal element
    that is not adjacent to $v$. It is straightforward to see that $a$ satisfies the third statement via axioms $\UMIII$
    and $\UMIV$. The case $E(w,v)$ and $\lnot E(u,v)$ can be handled symmetrically to find $a,b \in  V(\bA')$ that
    satisfy the fourth statement. 

    Denote by $f'$ the extension of $f$ defined by $v\mapsto v'$. If the first item holds, then $v$ is adjacent
    to every vertex in $V(\bA')$, and $v'$ is adjacent to every vertex $f(V(\bA'))$ (this follows from axiom 
    $\UMIV$). If the second statement holds, then it follows from the local betweenness axioms that
    $B$ is a (total) betweenness relation in $V(\bA')\cup\{v\}$, and in $f(V(\bA'))\cup\{v'\}$, so by
    axiom $\UMI$, $V(\bA')\cup\{v\}$ and $f(V(\bA'))\cup\{v'\}$ are independent sets (with respect to $E$).
    These arguments show that in the first and second item, the extension $f'$ preserves $E$ and $\lnot E$. 
    Combining both arguments, one can show that if  the third or fourth item holds, then $f'$ also preserves
    $E$ and $\lnot E$. The fact that $B(x,y,z)$ if and only if $B(f'(x), f'(y), f'(z))$
    follows from the definition of $u,w,a,b$ and Lemma~\ref{lem:unique-bet}. Finally, to see that $f'$ preserves
    the separation relation first identify the vertices in $V(\bA')\cup\{v\}$ with their images (under to $f'$).
    Then, let $S'$ be the separation relation defined  on $V(\bA')\cup\{v\}$ by this identification and the
    restriction of the separation relation in $\bB$ to $f[V(\bA')\cup\{v\}]$. Since $f[V(\bA')]$ is an independent set
    (with respect to $E$) and the restriction of $S$ and $S'$ to $f[V(\bA')]$ coincide, it follows from
    Observation~\ref{obs:sep-indeptendent} that $S = S'$, and hence $S(x_1,x_2,x_3,x_4)$ if and only if
    $S(f'(x_1), f'(x_2), f'(x_3), f'(x_4))$ for every $x_1,x_2,x_3,x_4\in V(\bA')\cup\{v\}$. Thus, $f'$ is an 
    embedding. 
\end{proof}

Subsets $U, U'\subseteq V(\mathbb A)$ are called \textit{$B$-disjoint}
if they are disjoint, $B(U)\cap U' = \varnothing$,  and $B(U')\cap U = \varnothing$. Moreover, 
we say that $k$ sets  $U_1,\dots, U_k$ are \emph{$B$-disjoint} if $U_i$ and $\bigcup_{j\neq i} U_j$
are $B$-disjoint for each $i\in\{1,\dots, k\}$ --- notice that this is a stronger condition
that being pairwise $B$-disjoint.

\begin{lemma}\label{lem:partition}
Let $\bA$ be an $\{E,B,S\}$-structure  satisfying the universal axioms. For every  $v\in V(\bA)$,
the relation 
$$\sim_v \; := \{(a,b) \in V \mid B(a,b,v) \lor B(b,a,v) \lor a = b\}$$
defines an equivalence relation on $V\setminus (N(v)\cup \{v\})$ with at most two equivalence classes $Y,Z$.
Moreover, $Y$ and $Z$ are independent sets, and  $N(v)$, $Y$, $Z$ are three $B$-disjoint sets. 
\end{lemma}
\begin{proof}
    Throughout the proof we interchange $B(a,b,c)$ and $B(c,b,a)$ (axiom $\UBI$). From the definition of $\sim_v$
    we know that it is a symmetric and reflexive relation. Now, suppose that $a\sim_v b$ and $b \sim_v c$, and
    notice that, up to symmetry, there are two non-trivial scenarios: $B(a,b,v)\land B(c,b,v)$ or $B(a,b,v)\land
    B(b,c,v)$. In the latter case, we use axiom $\UBIII$ to see that $B(a,c,v)$ so $a\sim_v c$. Now suppose that
    $B(v,b,a)\land B(v,b,c)$. By axiom $\UBIV$, we conclude that $B(v,b,c)$ or $B(v,c,d)$ so,
    $c\sim_v d$. Therefore, $\sim_v$ is an equivalence relation on $V\setminus (N(v)\cup\{v\})$.
    To see that it has at most two equivalence classes, consider
    $a,b,c\in V\setminus (N(v)\cup\{v\})$, and we observe that at least two of these belong to the  same $\sim_v$-class.
    Since $E$ has no triangles ($\UEII$), we assume without
    loss of generality that $\lnot E(a,b)$, so by axiom $\UMI$, either $B(a,b,v)$, $B(b,v,a)$, or $B(v,a,b)$.
    In the first and last case it immediately follows that $a \sim_v b$, so assume $B(b,v,a)$. In this case,
    axiom $\UMIV$ tells us that any neighbour of both $a$ and $b$ must be a neighbour of $v$, and thus,
    $\lnot E(c,a)$ or $\lnot E(c,b)$. Assume w.l.o.g.\ $\lnot E(c,a)$, so by axiom $\UMI$ either
    $B(a,c,v)$, $B(c,v,a)$, or $B(v,a,c)$. In the first and last case we obtain that $a\sim_v c$. Otherwise,
    we are in the case where  $B(a,v,c)$ and $B(a,v,b)$, and so, $\UBIV$ implies that 
    $B(b,c,v)$ or $B(c,b,v)$, i.e.,  $b\sim_v c$. Therefore, there are exactly two $\sim_v$-classes $Y$ and $Z$.
    Clearly, $Y$ and $Z$ are independent sets because of their definition and axiom $\UMI$. 

    To prove that $N(v), Y,Z$ are three $B$-disjoint sets we must prove that $N(v)$ and $(Y\cup Z)$ are $B$-disjoint,
    $Y$ and $N(v)\cup Z$ are $B$-disjoint, and $Z$ and $N(v)\cup Y$ are $B$-disjoint.
    Let $x_1,x_2\in N(v)$, and notice that by axiom
    $\UMIV$, every $u \in B(\{x_1,x_2\})$ is adjacent to $v$, and so, $B(N(v)) = N(v)$. In particular,
    $B(N(v)) \cap (Y\cup Z) = \varnothing$.
    Similarly, for every $w_1,w_2\in Y\cup Z = V\setminus (N(v)\cup\{v\})$
    and each $u \in B(\{w_1,w_2\})$, axiom $\UMIII$ tells us that $\lnot E(u,v)$ because
    $\lnot E(w_1,v)\land \lnot E(w_2,v)$. Therefore, $B(Y\cup Z)\setminus\{v\} = Y\cup Z$, 
    in particular, $B(Y \cup Z) \cap N(v) = \varnothing$ and so, $N(v)$ and $Y\cup Z$ are $B$-disjoint.
    Now, consider vertices $y_1,y_2\in Y$ and let $u\in B(\{y_1,y_2\})$. Assume without loss of generality that
    $B(v,y_1,y_2)$, so by axiom $\UBIII$ it is the case that $B(v,u,y_2)$, and $u\sim_v y_2$. Therefore, 
    $B(\{y_1,y_2\}) \subseteq Y$, and it follows that $B(Y) = Y$, analogously $B(Z) = Z$. In particular, 
    $B(Y) \cap (Z \cup N(v)) = \varnothing $. To see that $B(Z\cup N(v)) \cap Y = \varnothing$, let
    $w_1,w_2\in Z\cup N(v)$ and $u \in B(\{w_1,w_2\})$, and we see that $u \in Z \cup N(v)$.
    The cases $w_1,w_2 \in Z$ and $w_1,w_2\in N(v)$ were handled before. Also, if $E(v,u)$, then $u \in N(v)$,
    so we assume $B(w_1,u,w_2)$, $w_1 \in Z$, $w_2\in N(v)$, and $\lnot E(u,v)$. Since $\{w_1,u,v\}$ is
    an independent set of three vertices, either $B(w_1,u,v)$, $B(u,v,w_1)$, or $B(v,w_1,u)$. The first and
    last case imply that $u \sim_v w_1$ and so, $u\in Z$. The second case leads to a contradiction: $B(u,v,w_1)$
    and $B(w_2,u,w_1)$ imply $B(w_2,v,w_1)$ (axiom $\UBIII$), and this implies $\lnot E(v,w_2)$ (axiom $\UMI$),
    a contradiction to the definition of $w_2$. Therefore, $Y$ and $Z\cup N(v)$ are $B$-disjoint. Proving that
    $Z$ and $Y\cup N(v)$ are $B$-disjoint can be done analogously.
\end{proof}

\begin{corollary}\label{cor:def-XYZ}
    Let $\bA$ be a $\{E,B,S\}$-structure satisfying the universal axioms and let $\bA'$ be a finite
    substructure of $\bA$. Every $v\in V(\bA)\setminus V(\bA')$ defines a partition  of
    $V(\bA')$ into three (possibly empty) independent $B$-disjoint sets $X_v,Y_v,Z_v$ with the following properties.
    \begin{itemize}
        \item  $X_v\subseteq N(v)$ and $Y_v\cup Z_v\subseteq V\setminus (N(v)\cup\{v\})$, and $v\not\in B(X_v)\cup B(Y_v) \cup B(Z_v)$.
        \item  For any $y,y'\in Y_v$ either $B(y,y',v)$ or $B(y',y,v)$.
        \item  For any $z,z'\in Z_v$ either $B(z,z',v)$ or $B(z',z,v)$.
        \item  For any $y\in Y_v$ and $z\in Z_v$ either $E(y,z)$ or $B(y,v,z)$. 
    \end{itemize}
\end{corollary}

Notice that the sets $X_v, Y_v,Z_v$ defined in Corollary~\ref{cor:def-XYZ} also depend on 
$\bA$ and $\bA'$, but these will always be clear from context. 
Now, we aim for a result analogous to Lemma~\ref{lem:V-independent}
for the case where $V(\bA')$ is not an independent set. Unfortunately, we need
to proceed by considering several cases depending on which of the sets
$X_v,Y_v,Z_v$ are non-empty. To handle this, we propose three lemmas,
that altogether build-up to the aimed result.

It will be convenient to use the following shorthand notation for the previously mentioned
lemmas (and for the upcoming section). Consider $n$ variables $x_1,\dots, x_n$. By thinking
of these $n$ variables as points on the circle, we obtain a finite conjunct of 
atomic $\{S\}$-formulas defining a separation relation on $\{x_1,\dots, x_n\}$. If
the points are chosen so that when we start in $x_1$ and traverse the circle clockwise
we see $x_1,\dots, x_n$, then we denote  by $\Sep(x_1,\dots, x_n)$ the corresponding conjunction
of atomic $\{S\}$-formulas.  
In particular, if $S$ is a separation relation, then $\Sep(x_1,x_2,x_3,x_4)$ if and only if $S(x_1,x_2,x_3,x_4)$, 
and $\Sep(x_1,\dots, x_n)$ implies $\Sep(x_2,\dots, x_n,x_1)$ and 
$\Sep(x_1,\dots, x_{i-1}, x_{i+1}, \dots, x_n)$ for every $i \in \{2,\dots, n-1\}$. 
The following is also a straightforward observation implied by the universal axioms, in
particular, the local betweenness axioms, the separation axioms, and axioms $\UMI$ an
$\UMVIII$.

\begin{observation}\label{obs:Sep-y}
    Let $\bA$ be an $\{E,B,S\}$-structure satisfying the universal axioms. 
    Let $x_1,\dots, x_n \in V(\bA)$ be such that $\Sep(x_1,\dots, x_n)$ and $y\in V(\bA)$. If
    there is an $i\in\{1,\dots, n-1\}$ such that $B(x_i,y,x_{i+1})$ and $\lnot B(x_i,x_j,x_{i+1})$
    for every $j\in\{1,\dots, n\}$,  then $\Sep(x_1,\dots, x_i, y, x_{i+1}, \dots, x_n)$. 
\end{observation}

\begin{lemma}\label{lem:YZ}
Let $\bA$ and $\bB$ be $\{E,B,S\}$-structures satisfying the universal axioms. Let $\bA'$ be a finite substructure
    of $\bA$ with at least three vertices and at least one edge 
    and let $v\in V(\bA)\setminus V(\bA')$. If $X_v = \varnothing$, $Y_v\neq \varnothing$,
    $Z_v\neq \varnothing$, then (exactly) one of the following holds for some $y_1,y_2 \in Y_v$ and $z_1,z_2 \in Z_v$ such that
    $\{y_1,y_2\}$ is the pair of $B$-bounds of $Y_v$ and $\{z_1,z_2\}$ is the pair of $B$-bounds of $Z_v$.\footnote{The two pairs of $B$-bounds are unique, but the
    labeling is not, and in order to satisfy the items of this lemma, the labels must be chosen adequately.}
\begin{itemize}
    \item $Y_v = \{y_1\}$, $|Z_v| \ge 2$, 
    $y_1\notin N(v)$,
    $E(y_1,z_2)$, $B(v,z_1,z_2)$, and $S(y_1,v,z_1,z_2)$.
    \item $|Y_v| \ge 2$, $Z_v = \{z_1\}$, 
    $z_1 \notin N(v)$,
    $E(y_1,z_1)$,
    $B(y_1,y_2,v)$, and $S(y_2,v,z_1,y_1)$.
    \item $|Y_v|,|Z_v| \ge 2$, 
    $E(y_1,z_2)$, $B(y_1,y_2,v)$, $B(v,z_1,z_2)$, and $\Sep(y_1,y_2,v,z_1,z_2)$.
\end{itemize}
Moreover, in each of the cases above, if $f\colon \bA'\to \bB$ is an embedding and there is
$v'\in V(\bB)\setminus f[V(\bA')]$ such that $\bB$ models the same  formulas
substituting $v$ by $v'$ (and $y_1$, $y_2$, $z_1$, $z_2$ by their images under $f$),
then the extension of $f$  mapping $v$ to $v'$ is an embedding of the substructure of $\bA$ with vertex set
$V(\bA')\cup\{v\}$ into $\bB$.
\end{lemma}
\begin{proof}
    The three items are clearly disjoint cases.
    From the definition of $Y_v$ and $Z_v$ (Corollary~\ref{cor:def-XYZ}), 
    $\lnot E(v,y_i)$ and $\lnot E(v,z_j)$ hold for any $i,j\in\{1,2\}$. 
    Since $Y_v$ and $Z_v$ are subsets of different $\sim_v$ equivalence classes (see Lemma~\ref{lem:partition}),
    for every $y,y'\in Y_v$ with $B(y,y',v)$, and each $z\in Z_v$ we have  $\lnot B(y,z,v)$,
    and so,  axiom $\UMVIII$ implies $S(y,y',v,z)$. Similarly, for each $z,z'\in Z_v$
    with $B(y,z,z')$ and each $y\in Y_v$ we have $S(y,v,z,z')$. 

    Now, assume that $|Y_v|,|Z_v| \ge 2$. Let $y_1,y_2 \in Y_v$ be 
    such that $\{y_1,y_2\}$ is the pair of $B$-bounds of $Y_v$ and  $B(y_1,y_2,v)$, 
    and similarly let $z_1,z_2 \in Z_v$ be such that $\{z_1,z_2\}$ is the pair of $B$-bounds of $Z_v$
    such that $B(v,z_1,z_2)$. By the arguments in the previous paragraph we have
    $S(y_1,y_2,v,z_1)$, $S(y_1,y_2,v,z_2)$, and $S(y_1,v,z_1,z_2)$. Since $S$
    is a separation relation, it follows that $\Sep(y_1,y_2,v,z_1,z_2)$. It remains to verify $E(y_1,z_2)$. Since there are $y\in Y$ and $z\in Z$ with $E(y,z)$,
    we can assume by axiom $\UMIV$ that $y\in\{y_1,y_2\}$ and $z\in\{z_1,z_2\}$. In turn, 
    since $B(y_1,y_2,v)$ and $B(v,z_1,z_2)$, by the same axiom and $\lnot E(y_i,v)\land \lnot E(z_j,v)$
    for $i,j\in \{1,2\}$, we conclude $E(y_1, z_2)$. Therefore, if $|Y_v|,|Z_v| \ge 2$,
    then the third item holds. 
    Otherwise, since $\bA'$ has at least three vertices and $X_v = \emptyset$, we have that $|Y_v| = 1$ and $|Z_v| \geq 2$ or $|Y_v| \geq 2$ and $|\{Z_v\}|=1$. 
    These two cases can be proved with similar arguments.

    Finally, we prove that the extension $f'$ of $f$ defined by $v\mapsto v'$ is an embedding.  Recall
   that in any of the items  $y_1,y_2,z_1,z_2 \notin N(v)$. 
   Since $\{y_1,y_2\}$ is the pair of $B$-bounds of $Y_v$ and $f$ is an embedding, $\{f(y_1),f(y_2)\}$
   is the pair of $B$-bounds of $f[Y_v]$. Since $v'$ is neither adjacent to $f(y_1)$ nor to $f(y_2)$, axiom $\UMIII$ implies
    that $v'$ is not adjacent to any $y$ in $f[Y_v]$ --- analogously, $v'$ is not adjacent to any
    $z\in Z_v$. Hence, if follows from the definition of $Y_v$ and $Z_v$ that the extension $f'$
    preserves $E$ and $\lnot E$ (in any of the items). To see that $f$ preserves $B$ and $\lnot B$
    we first mention that it suffices to prove that it preserves $E$, $\lnot E$, and $B$ (see axioms
    $\UMI$, and local betweenness axioms). To see that $f'$ preserves $B$ first notice that
    $B$ it a (total) betweenness relation on $Y_v \cup\{v\}$ and on  $f[Y_v] \cup\{v'\}$, because
    these are independent sets of vertices. Hence, via Lemma~\ref{lem:unique-bet}, we see that $f'$
    preserves $B$ for triples $(a,b,c)\in (Y_v \cup\{v\})^3$, and one can argue similarly that
    $f'$ preserves $B$ for triples $(a,b,c)\in (Z_v \cup\{v\})^3$.  Recall that there is at least
    one edge $yz$ for $y\in Y_v$ and $z\in Z_v$, so $\lnot B(y,z,v)$ (axiom $\UMI$), and  thus
    $y$ and $z$ cannot belong to the same $\sim_v$-equivalence class  (see  Lemma~\ref{lem:partition}). 
    Similarly, $f(y)$ and $f(z)$ cannot belong to the same $\sim_{v'}$-equivalence class. Therefore, 
    $Y_v$ and $Z_v$, and $f[Y_v]$ and $f[Z_v]$ belong to different 
    equivalence
    classes of $\sim_v$ and of $\sim_{v'}$, respectively. This implies that for each $y\in Y$ and $z\in Z$, either $E(y,z)$ and $E(f(y), f(z))$,
    or $B(y,v,z)$ and $B(f(y), v', f(z))$, concluding that the extension $f'$ preserves $B$ and $\lnot B$.
    Finally, we argue that the extension $f'$ preserves $S$ and $\lnot S$, and we consider the third item. 
    First recall from the definition of $Y_v$ and $Z_v$ (Corollary~\ref{cor:def-XYZ}) and the choice of $y_1,y_2,z_1,z_2$
    that 
    \begin{itemize}
        \item $Y_v\subseteq B(\{y_1,y_2\})$ and $Z_v\subseteq B(\{z_1,z_2\})$,
        \item $v\not\in B(\{y_1,y_1\})
    \cup B(\{z_1,z_2\})$, and 
    \item $Y_v$ and $Z_v$ are $B$-disjoint. 
    \end{itemize}
    Moreover, since $f'$ preserves $B$, it is also
    the case that 
    \begin{itemize}
        \item 
        $f'[Y_v]\subseteq B(\{f'(y_1),f(y_2)\})$ and 
     $f'[Z_v]\subseteq B(\{f'(z_1),f'(z_2)\})$,
\item     $v'\not\in B(\{f'(y_1),f'(y_2)\}) \cup B(\{f'(z_1),f'(z_2)\})$,
    and 
    \item $f'[Y_v]$ and $f'[Z_v]$ are $B$-disjoint. 
    \end{itemize}
    Thus, we conclude that $f'$ preserves the separation
    relation by inductively applying Observation~\ref{obs:Sep-y} where the base case holds because
    $\Sep(y_1,y_2,v,z_1,z_2)$ and $\Sep(f'(y_1),f'(y_2),v',f'(z_1), f'(z_2))$. The first two items can be handled
    similarly. We thus conclude that $f'$ is an embedding.
\end{proof}

In the following lemma, we consider the case $X_v \neq \varnothing$, $Y_v\neq \varnothing$, and $Z_v =  \varnothing$.

\begin{lemma}\label{lem:XY}
Let $\bA$ and $\bB$ be $\{E,B,S\}$-structures satisfying the universal axioms. Let $\bA'$ be a finite substructure
    of $\bA$ with at least three elements and at least one edge, and $v\in V(\bA)\setminus V(\bA')$.
    If $X_v,Y_v,Z_v \neq \varnothing$, 
    then (exactly) one of the following holds 
    for some $x_1,x_2 \in X_v$ and $y_1,y_2 \in Y_v$
    such that $\{x_1,x_2\}$ is the pair of $B$-bounds of $X_v$ and $\{y_1,y_2\}$ is the pair of $B$-bounds of $Y_v$.
\begin{itemize}
    \item $|X_v| \ge 2$, $Y_v = \{y_1\}$, $x_1,x_2\in N(v)$, $y_1\not\in N(v)$,  $E(x_1,y_1)$, and $S(y_1,v,x_1,x_2)$.
    \item $X_v =\{x_1\}$, $|Y_v| \ge 2$, $x_1\in N(v)$, $E(x_1,y_2)$, and $B(y_1,y_2,v)$.
    \item $|X_v|,|Y_v| \ge 2$, $x_1,x_2\in N(v)$, $B(y_1,y_2,v)$, and $\Sep(x_1,x_2,y_1,y_2,v)$.
\end{itemize}
Moreover, in each of the cases above, if $f\colon \bA'\to \bB$ is an embedding and there is
$v'\in V(\bB)\setminus f[V(\bA')]$ such that $\bB$ models the same formulas substituting
$v$ by $v'$ (and $x_1$, $x_2$, $y_1$, $y_2$, by their images under $f$),
then the extension of $f$  mapping $v$ to $v'$ is an embedding of the substructure of $\bA$ with vertex set
$V(\bA')\cup\{v\}$ into $\bB$.
\end{lemma}
\begin{proof}
    Similarly as in Lemma~\ref{lem:YZ}, the three items are disjoint cases, and we first assume that $|Y_v| \ge 2$.
    Let $y_1,y_2 \in Y_v$ be such that the $\{y_1,y_2\}$ is the pair of $B$-bounds of $Y_v$ and 
    $B(y_1,y_2,v)$. Notice that if $E(x,y_1)$ for some $x\in X_v$, then axiom $\UMIV$ and $E(x,v)$
    imply $E(x,y_2)$. Also, $E(x,v)$ and axiom $\UMI$ imply $\lnot B(y_1,x,v)$, and this, together
    with $B(y_1,y_2,v)$ and axiom $\UMVIII$, implies $S(y_1,y_2,v,x)$ for every $x\in X_v$. Notice
    that at this point, it already follows that if $X_v =\{x_1\}$, $|Y_v| \ge 2$, then the second
    item holds. Suppose $|X_v| \ge 2$, and  let $x_1,x_2 \in X_v$ be 
    such that $\{x_1,x_2\}$ is the pair of $B$-bounds of $X_v$. 
    Since $S(y_1,y_2,v,x_1)$ and $S(y_1,y_2,v,x_2)$ we assume that $\Sep(y_1,y_2,v,x_1,x_2)$ (permute 
    the labels if needed), equivalently $\Sep(x_1,x_2,y_1,y_2,v)$. All is left to show now
    is $E(x_1,y_2)$, and since $E(x,y)$ for some $x\in X_v$ and $y\in Y_v$, we can assume by the previous
    argument that $E(x,y_2)$ for some $x\in X_v$. If $B(x_1,x,x_2)$, then axiom $\UMIII$ implies
    $E(x_1,y_2)\lor E(x_2,y_2)$. Assume $E(x_2,y_2)$, and notice that since $v$ is adjacent to
    $x_1$ and $x_2$, and $\lnot E(x_1,x_2)$ and $\lnot E(y_2,v)$, the vertices $x_1,x_2,y_2,v$ induce
    a $P_4$ or a $C_4$. Since $S(x_1,x_2,y_2,v)$, it follows via axiom $\UMVII$ that $E(x_1,y_2)$. 
    Therefore, $\Sep(x_1,x_2,y_1,y_2,v)$ and $E(x_1,y_2)$, and the remaining formulas of the third
    item hold by the definition of $y_1,y_2,x_1,x_2$ and of $X_v$ and $Y_v$.  Therefore, the
    lemma is true for the cases when    $|Y_v| \ge 2$. The case $|X_v| \ge 2$, $|Y_v| = \{y_1\}$
    follows by repeating some of the previous arguments: use axiom $\UMIII$ to show that
    $y_1$ is adjacent to one of $x,x' \in X_v$; notice that $x,x',v,y_1$
    must induce a $P_4$ or a $C_4$; use $\UMVII$ to argue that $S(x,x',y_1,v)$ and
    $E(y_1,x)$, or $S(x',x,y_1,v)$ and $E(y_1,x')$. Finally, label $x,x'$ by $x_1,x_2$ (according
    to the previous two possible cases) so that $S(x_1,x_2,y_1,v)$ and $E(y_1,x_1)$.
    We thus conclude that if $|X_v| \ge 2$ and $|Y_v| = 1$, then the first item holds.

    The moreover statement can be handled similarly as in Lemmas~\ref{lem:V-independent} and~\ref{lem:YZ};
    we give a brief overview. First, use axiom $\UMI$ to see that in any of the items above
    $y_1,y_2 \notin N(v)$, and so $E$ and $\lnot E$ are clearly preserved by axioms $\UMIII$ and $\UMIV$.
    Secondly $B$ and $\lnot B$ are preserved because $v$ and $v'$ are adjacent to every $x\in X_v$ and
    $x'\in f[X_v] = X_{v'}$, respectively; $Y_v$ and $Y_{v'}$ are independent sets with pairs of $B$-bounds $\{y_1,y_2\}$
    and $\{f(y_1),f(y_2)\}$, respectively; $B(y_1,y_2,v)$ and $B(f(y_1), f(y_2), v')$; now use axioms $\UMI$ and
    Lemma~\ref{lem:unique-bet} to conclude that $f$ preserves $B$ and $\lnot B$. Finally, it can be proved
    that $f'$ preserves the separation relation following a similar argument as in the last eight lines
    of the proof of Lemma~\ref{lem:YZ}:
    using the definitions of $x_1,x_2,y_1,y_2$ and of $X_v$ and $Y_v$ (Corollary~\ref{cor:def-XYZ}), 
    using the fact that $f'$ preserves $E$, $\lnot E$, $B$, and $\lnot B$, and inductively using
    Observation~\ref{obs:Sep-y} where the base case holds because $\Sep(x_1,x_2,y_1,y_2,v)$ and
    $\Sep(f'(x_1), f'(x_2), f'(y_1), f'(y_2), v')$.
    Therefore, $f'$ is an injective function preserving the separation relation, $B$, $\lnot B$, $E$ and $\lnot E$,
    i.e., an embedding.
\end{proof}

Finally, we consider the case when neither of the sets $X_v,Y_v,Z_v$ is the empty set. As we did in the previous two
lemmas, we distinguish between subcases depending on which of these sets have cardinality one.

\begin{lemma}\label{lem:XYZ}
Let $\bA$ and $\bB$ be $\{E,B,S\}$-structures satisfying the universal axioms. Let $\bA'$ be a finite substructure
    of $\bA$ and $v\in V(\bA)\setminus V(\bA')$. If $X_v \neq \varnothing$, $Y_v\neq \varnothing$, $Z_v\neq \varnothing$,
    then (exactly) one of the following holds 
    for some choice of $x_1,x_2 \in X_v$, $y_1,y_2 \in Y_v$, and $z_1,z_2 \in Z_v$ such that
   $\{x_1,x_2\}$ is the pair of $B$-bounds of $X_v$, $\{y_1,y_2\}$ is the pair of $B$-bounds of $Y_v$, and $\{z_1,z_2\}$ is the pair of $B$-bounds of $Z_v$. 
\begin{itemize}
    \item $X_v = \{x_1\}$, $Y_v = \{y_1\}$, $Z_v = \{z_1\}$, $x_1\in N(v)$, $y_1,z_1\not\in N(v)$,
    $S(x_1,y_1,v,z_1)$, and $x_1,y_1,z_1$ induce at least one edge.
    \item $|X_v| \ge 2$, $Y_v =\{y_1\}$, $Z_v = \{z_1\}$, $x_1,x_2\in N(v)$, $y_1,z_1\not\in N(v)$, 
    $\Sep(x_1,x_2,y_1,v,z_1)$, and each triple $x_1,y_1,z_1$ and $x_2,y_1,z_2$ induce at least
    one edge.
    \item $|X_v| = \{x_1\}$, $|Y_v| \ge 2$, $Z_v = \{z_1\}$,  $x_1\in N(v)$, $z_1\not\in N(v)$,
    $B(y_1,y_2,v)$, $\Sep(x_1, y_1,y_2, v, z_1)$, and each triple $x_1,y_1,z_1$ and $x_1,y_2,z_1$
    induce at least one edge.
    \item $|X_v| = \{x_1\}$, $Y_v  = \{y_1\}$, $|Z_v| \ge 2$, $x_1\in N(v)$, $y_1\not\in N(v)$,
    $B(v,z_1,z_2)$, $\Sep(x_1,y_1,v,z_1,z_2)$, and each triple $x_1,y_1,z_1$ and $x_1,y_1,z_2$
    induce at least one edge.
    \item $|X_v|,|Y_v| \ge 2$, $Z_v  = \{z_1\}$, $x_1,x_2\in N(v)$, $z_1\not\in N(v)$,
    $B(y_1,y_2,v)$, $\Sep(x_1,x_2,y_1,y_2,v,z_1)$, and for each $i,j\in\{1,2\}$ the triple $x_i,y_j,z_1$
    induces at least one edge.
    \item $|X_v|,|Z_v| \ge 2$, $Y_v  = \{y_1\}$, $x_1,x_2\in N(v)$, $y_1\not\in N(v)$,
    $B(v,z_1,z_2)$, $\Sep(x_1,x_2,y_1,v,z_1,z_2)$, and for each
    $i,k\in\{1,2\}$ the triple $x_i,y_1,z_k$ induces at least one edge.
    \item $X_v = \{x_1\}$, $|Y_v|,|Z_v| \ge 2$, $x_1\in N(v)$, 
    $B(y_1,y_2,v)$, $B(v,z_1,z_2)$, $\Sep(x_1,y_1, y_2,v,z_1,z_2)$,
    and for each $j,k\in\{1,2\}$ the triple $x_1,y_j,z_k$ induces at least one edge.
    \item $|X_v|,|Y_v|,|Z_v| \ge 2$, $x_1,x_2\in N(v)$, $B(y_1,y_2,v)$, $B(v,z_1,z_2)$,
    $\Sep(x_1,x_2,y_1, y_2,v,z_1,z_2)$, and for each $i,j,k\in\{1,2\}$ the triple $x_i,y_j,z_k$ induces at least one edge.
\end{itemize}
Moreover, in each of the cases above, if $f\colon \bA'\to \bB$ is an embedding and there is
$v'\in V(\bB)\setminus f[V(\bA')]$ such that $\bB$ models the same formulas substituting
$v$ by $v'$  (and $x_1$, $x_2$, $y_1$, $y_2$, $z_1$, $z_2$ by their
images under $f$), then the extension of $f$  mapping $v$ to $v'$ is an embedding of the substructure of
$\bA$ with vertex set $V(\bA')\cup\{v\}$ into $\bB$.
\end{lemma}
\begin{proof}
    Let $x\in X_v$, $y\in Y_v$,  $z\in Z_v$ and notice that if there is no edge induced by
    the triple $x,y,z$, then axiom $\UMI$ implies that $B(x,y,z)$, $B(y,z,x)$ or $B(z,x,y)$.
    But either of this contradicts the fact that $X_v$, $Y_v$, $Z_v$ are $B$-disjoint sets
    (see Corollary~\ref{cor:def-XYZ}, and the definition of $B$-disjoint sets above Lemma~\ref{lem:partition}).
    This means that in every item and for any choice of $x_1,x_2,y_1,y_2,z_1,z_2$,  the triple $x_i,y_j,z_k$
    induces at least one edge for every for each $i,j,k\in\{1,2\}$. Also, from the definition of 
    $X_v,Y_v,Z_v$ in every item the formulas $E(v,x_i)$, $\lnot E(v,y_j)$, $\lnot E(v,z_k)$ hold
    for $i,j,k\in\{1,2\}$.

    Now we show that $S(x,y,v,z)$ for every $x\in X_v$, $y\in Y_v$, and  $z\in Z_v$. By the last item
    of Corollary~\ref{cor:def-XYZ}, either $E(z,y)$ or $B(y,v,z)$. Since $X_v, Y_V, Z_v$ are $B$-disjoint,
    it follows that
    $\lnot B(y,x,z)$ and so, if $B(y,v,z)$, them axiom $\UMVIII$ implies $S(x,y,v,z)$. Otherwise, 
    if $E(z,y)$, since $E(v,x)$, $\lnot E(v,y)$, and $\lnot E(v,z)$,  it must be the case that 
    $x,y,v,z$ induce a $P_4$ or $2K_2$ because $\bA$ is triangle-free (axiom $\UEII$). Thus, it follows
    from axiom $\UMVII$ and $\lnot E(v,y)\land \lnot E(v,z)$ that $S(x,y,v,z)$.
    
    Notice that the first two paragraphs already prove that the first item holds whenever $|X_v| =
    |Y_v| = |Z_v| = 1$. Now we consider the case $|X_v|, |Y_v|, |Z_v| \ge 2$, and let $y_1,y_2 \in Y_v$ be
    such that $\{y_1,y_2\}$ is the pair of $B$-bounds of $Y_v$ and $B(y_1,y_2,v)$, similarly, let $z_1,z_2 \in Z_v$ be such that
    $\{z_1,z_2\}$ is the pair of $B$-bounds of $Z_v$ and 
    that $B(v,z_1,z_2)$, and $x,x' \in X_v$ such that $\{x,x'\}$ is the pair of $B$-bounds of $X_v$. By the second paragraph and the fact
    that $S$ is a separation relation, it must be the case that either $\Sep(x,x',y_1,v,z_2)$
    or $\Sep(x',x,y_1,v,z_2)$, so relabel $x$ and $x'$ to $x_1$ and $x_2$ so that $\Sep(x_1,x_2,y_1,v,z_2)$.
    Notice that $\lnot B(y_1,x_1,v)$ and $\lnot B(y_1,x_2,v)$ (because $E(v,x_1)$ and $E(v,x_2)$,
    and axiom $\UMI$), and $\lnot B(y_1,z_2,v)$ (from the definition $Y_v$ and $Z_v$ in
    Corollary~\ref{cor:def-XYZ} and from definition of the equivalence relation $\sim_v$ in Lemma~\ref{lem:partition}). 
    Since $B(y_1,y_2,v)$, $\lnot B(y_1,x_1,v)$, $\lnot B(y_1,x_2,v)$, $\lnot B(y_1,z_2,v)$,
    and $\Sep(x_1,x_2,y_1,v,z_2)$ we see via Observation~\ref{obs:Sep-y} that
    $\Sep(x_1,x_2,y_1,y_2,v,z_2)$. It can be proved with  similar arguments that
    $\Sep(x_1,x_2,y_1,y_2,v,z_1,z_2)$. This, together with $B(y_1,y_2,v)$, $B(v,z_1,z_2)$
    (as previously mentioned), and the arguments in the first paragraph, implies that if $|X_v|, |Y_v|, |Z_v| \ge 2$,
    then the last item holds. The remaining cases where at least one and at most two of $X_v,Y_v,Z_v$ contain
    exactly one element  be handled similarly. 

    The moreover statement follows with very similar arguments as in the proofs of Lemmas~\ref{lem:YZ} and~\ref{lem:XY}.
    As in Lemma~\ref{lem:YZ}, notice that $B(y_1,y_2,v)$ and $B(v,z_1,z_2)$ imply $y_1,y_2,z_1,z_2 \notin N(v)$ 
    via axiom $\UMI$. 
    The extension $f'$ preserves $E$ and $\lnot E$: for vertices in $Y_v\cup Z_y$  mimic the corresponding argument
    in Lemma~\ref{lem:YZ}, and for vertices in $X_v$ see the argument in Lemma~\ref{lem:XY}.
    To see that $f'$ preserves the local betweenness relation, it suffices to verify that it preserves $B$ for
    triples $a,b,c$ with $v\in\{a,b,c\}$ and $b,c\in Y_v\cup Z_v$, and this can be done with exactly
    the same argument as in the proof of Lemma~\ref{lem:YZ}. To see that $f'$ preserves the separation
    relation, imitate (again) the  inductive argument in Lemma~\ref{lem:YZ}, where now the base
    case is given by the assumptions $\Sep(x_1, x_2, y_1, y_2, v, z_1, f'(z_2))$ and
    $\Sep(f'(x_1), f'(x_2), f'(y_1), f'(y_2), v', f'(z_1), f'(z_2))$. We conclude that the extension $f'$
    is an embedding.
\end{proof}

\subsection{Extension axioms}
The aim of the extension axioms (listed below) is to capture the  structure $(\mathbb C_3, B, S)$ up to isomorphism.
Intuitively speaking, the extension   axioms state that  for every partition of $V(\bA')$ 
as described in Lemmas~\ref{lem:V-independent}, \ref{lem:YZ}, \ref{lem:XY}, 
and~\ref{lem:XYZ}
there is
a vertex $v'\in V(\bA')\setminus V(\bA)$ which is \textit{nicely} related to this partition. Then, we 
inductively embed any structure satisfying the universal axioms into a structure satisfying
the universal and extension axioms (Theorem~\ref{thm:extending-emb}).

Similarly as above, all variables are implicitly universally quantified, unless they are explicitly existentially
quantified. Moreover, given  a (negated) atomic formula $\phi$, we use 
$\phi(x_i,\overline y)$ as a shorthand for $\phi(x_1,\overline y) \land \phi(x_2, \overline y)$.
For instance,  $E(x_i,y_1) \land z_i \neq y_j$ says that $y_1$ is not adjacent to neither $x_1$ nor
$x_2$, and $y_1\neq z_1$, $y_1\neq z_2$, $y_2\neq z_1$, and $y_2 \neq z_2$. Finally, recall that 
$\Sep$ is not an atomic formula, but for every tuple of variables $\overline x$, the formula $\Sep(\overline x)$ denotes
a (finite) conjunction of atomic $\{S\}$-formulas where the number of conjuncts depends on the length of $\overline x$.

\vspace{0.4cm}
\noindent \textit{Extension axioms} (illustrated in Figures~\ref{fig:ex-axioms-I} and~\ref{fig:ex-axioms-II}; white (resp.\ filled) vertices depict universally (resp.\ existentially)
quantified vertices,  the betweenness relation $B$ is the usual betweenness relation on the (dotted) line, and the separation relation
is the natural separation relation on the (dotted) circumference.)

\vspace{0.15cm}
  $\EX$ There exists a vertex. 

  \vspace{0.25cm}
 $\EXIa$ If $\lnot E(x_1,x_2)$, then there are $v_1,v_2$ such that $E(v_1,x_i)$ and $\lnot E(v_2,x_i)$.

 $\EXIb$ If $\lnot E(y_1,y_2)$ and $y_1 \neq y_2$, then $\exists v_1,v_2,v_3$ such that $B(v_1,y_1,y_2)$, $B(y_1,v_2,y_2)$, and 
 
 \phantom{$\EXIb$} $B(y_1,y_2,v_3)$.

 $\EXIc$ If $B(x_i,y_1,y_2)\land B(x_1,x_2,y_i)$, then there is a vertex $v$ such that $E(v,x_i)$ and $B(y_1,y_2,v)$.

\vspace{0.25cm}
 $\EXIIa$ If $E(y,z)$, then there is a vertex $v$ such that $\lnot E(v,y)\land \lnot E(v,z)$.
 
  $\EXIIb$ If $E(y,z_2)$, $\lnot E(z_1,z_2)$, then $\exists v$ such that $\lnot E(v,y)$, $\lnot E(v,z_i)$, $B(v,z_1,z_2)$, and $S(y,v,z_1,z_2)$.

 $\EXIIc$ If $S(y_1,y_2,z_1,z_2)$, $\lnot E(y_1,y_2)$, $\lnot E(z_1,z_2)$, and $E(y_1,z_2)$, then $\exists v$ such that $\lnot E(y_i,v)$, 
 
 \phantom{$\EXIIc$}  $\lnot E(z_j,v)$, $B(y_1,y_2,v)$, $B(v,z_1,z_2)$, and $\Sep(y_1,y_2,v,z_1,y_1)$.

\vspace{0.25cm}
$\EXIIIa$ If $E(x,y)$, then there is a vertex $v$ such that $E(v,x)\land \lnot E(v,y)$.

$\EXIIIb$ If $\lnot E(x_1,x_2)$, $E(x_1,y)$, and $x_1\neq x_2$, then $\exists v$ such that $E(v,x_i)$, $\lnot E(x,y)$ and $S(y,v,x_1,x_2)$.

$\EXIIIc$ If $E(x,y_2)$, $\lnot E(y_1,y_2)$, and $y_1\neq y_2$, then $\exists v$ such that $E(v,x)$ and $B(y_1,y_2,v)$.

$\EXIIId$ If $\lnot E(x_1,x_2)$, $\lnot E(y_1,y_2)$,  $E(x_1,y_2)$, $S(x_1,x_2,y_1,y_2)$, then $\exists v$ such that $E(v,x_i)$, $\lnot E(v,y_j)$, 

\phantom{$\EXIIId$} $B(y_1,y_2,v)$, and $\Sep(x_1,x_2,y_1,y_2,v)$.

\vspace{0.25cm}
$\EXIVa$ If $|\{x,y,z\}| = 3$ and $x,y,z$ induce at least one edge, then $\exists v$ such that $E(v,x)$, $\lnot E(v,y)$,

\phantom{$\EXIVa$} $\lnot E(v,z)$ and $S(x,y,v,z)$.

$\EXIVb$ If $S(x_1,x_2,y,z)$, $\lnot E(x_1,x_2)$, and $x_i,y,z$ induce at least one edge, then  $\exists v$ such that

\phantom{$\EXIVb$} $E(v,x_i)$, $\lnot E(v,y)$, $\lnot E(v,z)$, and $\Sep(x_1,x_2,y,v,z)$.

$\EXIVc$ If $S(x,y_1,y_2,z)$, $\lnot E(y_1,y_2)$,  and $x,y_j,z$ induce at one edge, then $\exists v$ such that

\phantom{$\EXIVc$}  $E(v,x)$, $\lnot E(v,z)$, $B(y_1,y_2,v)$, and  $\Sep(x,y_1,y_2,v,z)$.

$\EXIVd$ If $\Sep(x_1,x_2,y_1,y_2,z)$, $\lnot E(x_1,x_2)$, $\lnot E(y_1,y_2)$, and $x_i,y_j,z_k$ induce at least one edge, 

\phantom{$\EXIVd$}  then $\exists v$ such that $E(v,x_i)$, $\lnot E(v,z)$, $B(y_1,y_2,v)$, and $\Sep(x_1,x_2,y_1,y_2,v,z)$.

$\EXIVe$ If  $\Sep(x,y_1,y_2,z_1,z_2)$, $\lnot E(y_1,y_2)$, $\lnot E(z_1,z_2)$, and $x_i,y_j,z_k$ induce at at least one edge,

\phantom{$\EXIVe$} then  $\exists v$ such that $E(v,x)$, $B(y_1,y_2,v)$, $B(v,z_1,z_2)$, and $\Sep(x,y_1,y_2,v,z_1,z_2)$.

$\EXIVf$ If  $\Sep(x_1,x_2,y_1,y_2,z_1,z_2)$, $\lnot E(x_1,x_2)$, $\lnot E(y_1,y_2)$,  $\land \lnot E(z_1,z_2)$, and $x_i,y_j,z_k$
induce at 

\phantom{$\EXIVf$}  least one edge, then $\exists v$ such that $E(v,x_i)$, $B(y_1,y_2,v)$, $B(v,z_1,z_2)$, and \hfill

\phantom{$\EXIVf$} $\Sep(x_1,x_2,y_1,y_2,v,z_1,z_2)$.

\begin{observation}\label{obs:ex-ax}
Similarly to the case of the universal axioms, it follows from the geometric construction that
$(\mathbb C_3, B,S)$ satisfies the listed existential axioms. 
\end{observation}

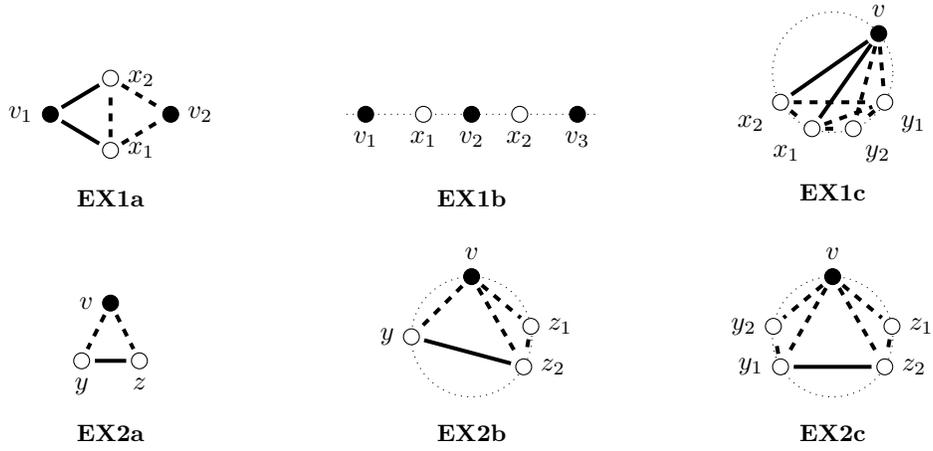
\begin{figure}[ht!]
\centering
\begin{tikzpicture}[scale = 0.8]


  \begin{scope}[scale = 0.8]
    \node (L2) at (0,-1.5)  {\small $\EXIa$ };
        \node [vertex, fill = black, label = left:{$v_1$}] (v1) at (-1.25,0.25) {};
        \node [vertex, fill = black, label = right:{$v_2$}] (v2) at (1.25,0.25) {};
        \node [vertex, label = right:{$x_1$}] (x1) at (0,-0.5) {};
        \node [vertex, label = right:{$x_2$}] (x2) at (0,1) {};

        \foreach \from/\to in {v1/x1, v1/x2} 
        \draw [edge] (\from) to (\to);

        \foreach \from/\to in {x1/x2, x1/v2, x2/v2} 
        \draw [edge, dashed] (\from) to (\to);
  \end{scope}

  \begin{scope}[xshift = 6cm, scale = 0.8]
         \node (L2) at (0,-1.5)  {\small $\EXIb$};
        \draw[dotted]  (-2.6,0.25) to (2.6,0.25);
        \node [vertex, fill = black, label = below:{$v_2$}] (v) at (0,0.25) {};
        \node [vertex, fill = black, label = below:{$v_1$}] (1) at (-2.2,0.25) {};
        \node [vertex, label = below:{$x_1$}] (2) at (-1,0.25) {};
        \node [vertex, label = below:{$x_2$}] (3) at (1,0.25) {};
        \node [vertex, fill = black, label = below:{$v_3$}]  (4) at (2.2,0.25) {};

        
  \end{scope}

   \begin{scope}[xshift = 12cm, yshift = 0.9cm, scale = 0.4]
        \node (L2) at (0,-5)  {\small $\EXIc$};
        \draw [black, dotted] (0,0) circle [radius = 2.5];
        \node [vertex, label = 210:{$x_2$}] (x1) at (210:2.5){};
        \node [vertex, label = 250:{$x_1$}] (x2) at (250:2.5){};
        \node [vertex, fill = black, label = above:{$v$}] (v) at (40:2.5){};
        
        \node [vertex, label = 330:{$y_1$}] (y1)  at (330:2.5){};
        \node [vertex, label = 290:{$y_2$}] (y2) at (290:2.5){};

        \foreach \from/\to in {v/x2, v/x1} 
        \draw [edge] (\from) to (\to);
        \foreach \from/\to in {x1/x2, y1/y2, v/y1, v/y2, x1/y1, x2/y2, x2/y1} 
        \draw [edge, dashed] (\from) to (\to);
        
  \end{scope}


  

  \begin{scope}[yshift = -3.9cm, scale = 0.8]
    \node (L2) at (0,-1.5)  {\small $\EXIIa$ };
        \node [vertex, fill = black, label = left:{$v$}] (v) at (0,1.2) {};
        \node [vertex, label = below:{$y$}] (y) at (-0.6,0) {};
        \node [vertex, label = below:{$z$}] (z) at (0.6,0) {};

        \draw [edge] (y) to (z);

        \foreach \from/\to in {v/z, v/y} 
        \draw [edge, dashed] (\from) to (\to);
  \end{scope}
  
  \begin{scope}[yshift = -3.5cm, xshift = 6cm, scale = 0.4]
        \node (L2) at (0,-4)  {\small $\EXIIb$};
    
        \draw [black, dotted] (0,0) circle [radius = 2.5];
        \node [vertex, label = left:{$y$}] (y) at (180:2.5){}; 
        \node [vertex, fill = black, label = above:{$v$}] (v) at (90:2.5){};
        
        \node [vertex, label = right:{$z_1$}] (z1)  at (10:2.5){};
        \node [vertex, label = right:{$z_2$}] (z2) at (330:2.5){};

        \foreach \from/\to in {y/z2} 
        \draw [edge] (\from) to (\to);
        \foreach \from/\to in {z1/z2, v/y, v/z1, v/z2} 
        \draw [edge, dashed] (\from) to (\to);
  \end{scope}

  \begin{scope}[xshift = 12cm, yshift = -3.5cm, scale = 0.4]
       \node (L2) at (0,-4)  {\small $\EXIIc$};
    
        \draw [black, dotted] (0,0) circle [radius = 2.5];
        \node [vertex, label = left:{$y_1$}] (y1) at (210:2.5){};
        \node [vertex, label = left:{$y_2$}] (y2) at (170:2.5){};
        \node [vertex, fill = black, label = above:{$v$}] (v) at (90:2.5){};
        
        \node [vertex, label = right:{$z_1$}] (z1)  at (10:2.5){};
        \node [vertex, label = right:{$z_2$}] (z2) at (330:2.5){};
        
        \draw [edge] (y1) to (z2);
        \foreach \from/\to in {y1/y2, z1/z2, v/y1, v/y2, v/z1, v/z2} 
        \draw [edge, dashed] (\from) to (\to);
    \end{scope}
 
\end{tikzpicture}
\caption{A depiction of the first six extension axioms. 
}
\label{fig:ex-axioms-I}
\end{figure}


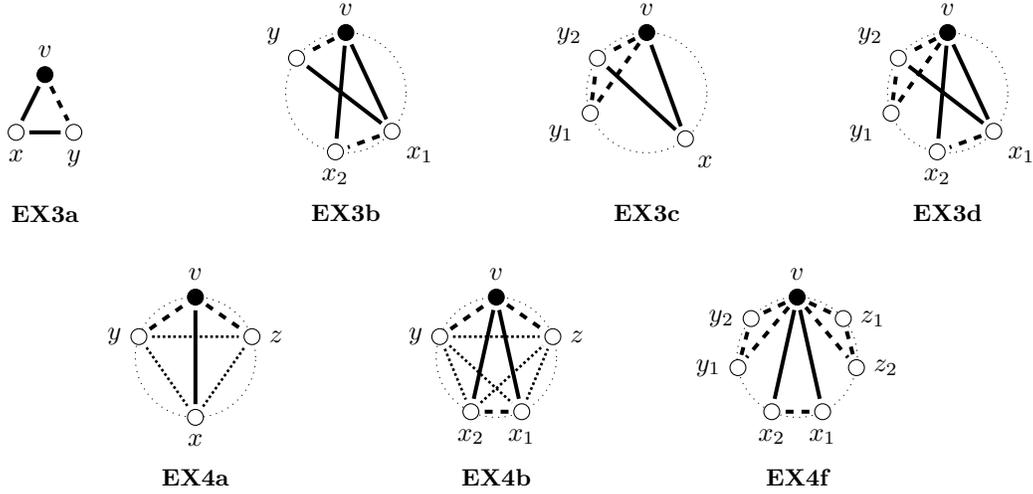
\begin{figure}[ht!]
\centering
\begin{tikzpicture}[scale = 0.8]


  \begin{scope}[scale = 0.8]
    \node (L2) at (0,-1.4)  {\small $\EXIIIa$ };
        \node [vertex, fill = black, label = above:{$v$}] (v) at (0,1.5) {};
        \node [vertex, label = below:{$x$}] (x) at (-0.6,0.3) {};
        \node [vertex, label = below:{$y$}] (y) at (0.6,0.3) {};

        \draw [edge, dashed] (v) to (y);

        \foreach \from/\to in {v/x, x/y} 
        \draw [edge] (\from) to (\to);
  \end{scope}

  \begin{scope}[xshift = 5cm, yshift = 0.9cm, scale = 0.4]
        \node (L2) at (0,-5)  {\small $\EXIIIb$};
        \draw [black, dotted] (0,0) circle [radius = 2.5];
        \node [vertex, label = 310:{$x_1$}] (x1) at (320:2.5){}; 
        \node [vertex, label = below:{$x_2$}] (x2) at (260:2.5){};
        
        \node [vertex, fill = black, label = above:{$v$}] (v) at (90:2.5){};
 
        \node [vertex, label = 145:{$y$}] (y) at (145:2.5){};

        \foreach \from/\to in {x1/y, x1/v, x2/v} 
        \draw [edge] (\from) to (\to);
        \foreach \from/\to in {x1/x2, v/y} 
        \draw [edge, dashed] (\from) to (\to);
        
  \end{scope}

  \begin{scope}[yshift = 0.9cm, xshift = 10cm, scale = 0.4]
        \node (L2) at (0,-5)  {\small $\EXIIIc$};
    
        \draw [black, dotted] (0,0) circle [radius = 2.5];
        \node [vertex, label = 290:{$x$}] (x) at (310:2.5){}; 
        \node [vertex, fill = black, label = above:{$v$}] (v) at (90:2.5){};
        
        \node [vertex, label = 145:{$y_2$}] (y2)  at (145:2.5){};
        \node [vertex, label = 200:{$y_1$}] (y1) at (200:2.5){};

        \foreach \from/\to in {x/y2, x/v} 
        \draw [edge] (\from) to (\to);
        \foreach \from/\to in {y1/y2, v/y1, v/y2} 
        \draw [edge, dashed] (\from) to (\to);
  \end{scope}

  \begin{scope}[xshift = 15cm, yshift = 0.9cm, scale = 0.4]
        \node (L2) at (0,-5)  {\small $\EXIIId$};
        \draw [black, dotted] (0,0) circle [radius = 2.5];
        \node [vertex, label = 310:{$x_1$}] (x1) at (320:2.5){}; 
        \node [vertex, label = below:{$x_2$}] (x2) at (260:2.5){};
        
        \node [vertex, fill = black, label = above:{$v$}] (v) at (90:2.5){};
 
        \node [vertex, label = 145:{$y_2$}] (y2)  at (145:2.5){};
        \node [vertex, label = 200:{$y_1$}] (y1) at (200:2.5){};

        \foreach \from/\to in {x1/y2, x1/v, x2/v} 
        \draw [edge] (\from) to (\to);
        \foreach \from/\to in {x1/x2, y1/y2, v/y1, v/y2} 
        \draw [edge, dashed] (\from) to (\to);
  \end{scope}
  

  \begin{scope}[yshift = -3.5cm, xshift = 2.5cm, scale = 0.4]
        \node (L2) at (0,-5)  {\small $\EXIVa$};
        \draw [black, dotted] (0,0) circle [radius = 2.5];
        \node [vertex, label = below:{$x$}] (x) at (270:2.5){};
        \node [vertex, label = left:{$y$}] (y)  at (160:2.5){};
        \node [vertex, label = right:{$z$}] (z) at (20:2.5){};
    
        \node [vertex, fill = black, label = above:{$v$}] (v) at (90:2.5){};
        
        \foreach \from/\to in {v/z, y/v} 
        \draw [edge, dashed] (\from) to (\to);
        \foreach \from/\to in {x/z, y/x, y/z} 
        \draw [edge, line width = 1pt, densely dotted] (\from) to (\to);
        \draw [edge] (x) to (v);
  \end{scope}
  
  \begin{scope}[yshift = -3.5cm, xshift = 7.5cm, scale = 0.4]
        \node (L2) at (0,-5)  {\small $\EXIVb$};
        \draw [black, dotted] (0,0) circle [radius = 2.5];
        \node [vertex, label = below:{$x_1$}] (x1) at (295:2.5){};
        \node [vertex, label = below:{$x_2$}] (x2) at (245:2.5){};
        \node [vertex, label = left:{$y$}] (y)  at (160:2.5){};
        \node [vertex, label = right:{$z$}] (z) at (20:2.5){};
    
        \node [vertex, fill = black, label = above:{$v$}] (v) at (90:2.5){};
        
        \foreach \from/\to in {v/z, y/v, x1/x2} 
        \draw [edge, dashed] (\from) to (\to);
        \foreach \from/\to in {x1/v, x2/v} 
        \draw [edge] (\from) to (\to);
        \foreach \from/\to in {x1/z, y/x1, x2/z, y/x2, y/z} 
        \draw [edge, line width = 1pt, densely dotted] (\from) to (\to);
  \end{scope}

  \begin{scope}[xshift = 12.5cm, yshift = -3.5cm, scale = 0.4]
        \node (L2) at (0,-5)  {\small $\EXIVf$ };
        \draw [black, dotted] (0,0) circle [radius = 2.5];
        \node [vertex, label = below:{$x_1$}] (x1) at (295:2.5){};
        \node [vertex, label = below:{$x_2$}] (x2) at (245:2.5){};
        \node [vertex, fill = black, label = above:{$v$}] (v) at (90:2.5){};
        
        \node [vertex, label = left:{$y_1$}] (y1)  at (190:2.5){};
        \node [vertex, label = left:{$y_2$}] (y2) at (140:2.5){};

        \node [vertex, label = right:{$z_2$}] (z2) at (-10:2.5){};
        \node [vertex, label = right:{$z_1$}] (z1) at (40:2.5){};
        
        \foreach \from/\to in {v/x2, v/x1} 
        \draw [edge] (\from) to (\to);
        \foreach \from/\to in {x1/x2, z1/z2, y1/y2, v/z1, v/z2, v/y1, v/y2} 
        \draw [edge, dashed] (\from) to (\to);
  \end{scope}

\end{tikzpicture}
\caption{A depiction of seven of the last ten extension axioms: 
Every triple $x,y,z$ (with possible subindices) induces at least
one edge; we represent this in $\EXIVa$, $\EXIVb$ with densely dotted triangles. Since such triangles make the picture of $\EXIVf$
too messy, we choose not to draw them.
}
\label{fig:ex-axioms-II}
\end{figure}

\medskip 

This list of axioms is not minimal, meaning that some are implied by others. We choose to present it including these
logical redundancies so that it becomes evident where these axioms come from: all possible cases in Lemmas~\ref{lem:V-independent},
\ref{lem:YZ}, \ref{lem:XY},  and~\ref{lem:XYZ} (up to symmetry between $Y_v$ and $Z_v$). Also, in this way we avoid extra technical
proofs that provide no new ideas nor interesting insight for the present work. For instance, we claim without a proof that axioms
$\EXIVa, \EXIVb$, and $\EXIVf$, imply axioms $\EXIVc, \EXIVd$, and $\EXIVe$.

\begin{theorem}\label{thm:extending-emb}
    Let $\bA$ and $\bB$ be two countable $\{E,B,S\}$-structures satisfying the universal axioms, and let
    $f \colon \bA' \to \bB$ be an embedding of a finite 
    substructure $\bA'$ of $\bA$. 
    If $\bB$ satisfies the extension axioms, then $f$ can be extended to an embedding from $\bA$ into $\bB$.
\end{theorem}
\begin{proof}
    We construct the embedding inductively. Let $a_1,a_2,\dots$ be an enumeration of $V(\bA)$, and
    let $i \in {\mathbb N}$ be smallest such that $a_i$ is not in the domain $V(\bA')$ of $f$
    (if no such $i$ exists, then $\bA' = \bA$ and $f$ is an embedding from $\bA$ into $\bB$).
    We have to find a vertex $b \in V(\bB)$ such that the extension $f_i$ of $f$ that maps
    $a_i$ to $b$ is an isomorphism between the finite substructure of $\bA$ with vertex set
    $V(\bA')\cup \{a_i\}$ and the finite substructure of $\bB$ with vertex set $V(\bB')\cup \{b\}$.
    The case $|V(\bA')|  \le 2$ follows by either of $\EX$, $\EXIa$, $\EXIb$, $\EXIIa$, or $\EXIIIa$ ---
    from here onward we assume that $|V(\bA')| \ge 3$.

    We first consider the case when $V(\bA')$ is an independent set with respect to $E$. Let $u,w$
    be the $B$-bounds of $V(\bA')$. We proceed by considering the subcases described in
    Lemma~\ref{lem:V-independent}, and in each case we define $b\in V(\bB)$ so that it
    satisfied the same item with respect to $f[V(\bA')]$: if $E(a_i, u)$ and $E(a_i,w)$, let $b\in V(\bB)$ 
    be a neighbour of $f(u)$ and $f(w)$ (axiom $\EXIa$); in the second  item
    of Lemma~\ref{lem:V-independent}, if $B(a_i,u,w)$ (resp.\ $B(u,w,a_i)$), then we use axiom $\EXIb$ applied to 
    $y_1 = f(u)$, $y_2 = f(w)$, and let $b = v_1$ (resp.\ let $b = v_3$); also in the second item, if $B(a,a_i,b)$,
    let $b$ be the vertex  $v_2$ obtained from axiom $\EXIb$ for $y_1 = f(a)$ and $y_2 = (b)$; the last two items
    can be handled similarly  via axiom $\EXIc$.
    The fact that the extension $f_i$ is an embedding follows from the definition of $b$ and the ``morevover''
    statement in Lemma~\ref{lem:V-independent}.
    
    Now, suppose that $\bA'$ contains at least one edge, and consider the partition $(X,Y,Z)$ defined by $a_i$
    on $V(\bA)$ given by Corollary~\ref{cor:def-XYZ}.   We proceed over a series of case distinctions that depend on
    which of these are non-empty sets. Since each of $X,Y,Z$ is independent (with respect to $E$), at most one
    of $X, Y, Z$ is the empty set. The arguments are essentially the same as in the last paragraph, so we only
    mention which lemma from the universal section should be used, and which extension axiom in each item of the 
    corresponding lemma. 

    \begin{itemize}
        \item $X = \varnothing$, $Y\neq \varnothing$, and $Z\neq \varnothing$.  This case corresponds to Lemma~\ref{lem:YZ}.
        In the first and second items find $b\in V(\bB)$ via axiom $\EXIIb$ (the second case is symmetric to the first one
        by permuting $Y$ and $Z$, and relabeling $(y_1,z_1,z_2)\mapsto (z_1,y_2,y_1)$). In the third item find 
        such a vertex $b$ by axiom $\EXIIc$. The moreover statement of Lemma~\ref{lem:YZ} now proves the claim in this case.
        
        \item $X \neq \varnothing$, $Y\neq \varnothing$, and $Z =  \varnothing$.  In this case, we proceed via Lemma~\ref{lem:XY}, where
        in the first item find $b\in V(\bB)$ by axiom $\EXIIIb$, in the second one by axiom $\EXIIIc$, and in the third one by axiom $\EXIIId$.
        The fact that the extension $f_i$ mapping $a_i$ to $b$ is an embedding follows from the moreover statement of Lemma~\ref{lem:XY}.
        The case $X \neq \varnothing$, $Y = \varnothing$, and $Z\neq \varnothing$ can be argued with symmetric arguments.

        \item $X \neq \varnothing$, $Y\neq \varnothing$, and $Z\neq \varnothing$. This final case corresponds to Lemma~\ref{lem:XYZ}. 
        Notice that item 3 and 4 are symmetric (go from 3 to 4 via $(Y,Z,y_1,y_2,z_1)\mapsto (Z,Y,z_2,z_1,z_1)$),
        and so are cases 5 and 6 (go from 5 to 6 via $(Y,Z, y_1,y_2,z_1,z_2)\mapsto (Z,Y, z_2, z_1, y_2, y_1)$). Hence, we only consider
        items 1, 2, 3, 5, 7, and 8, and in each of these find $b\in V(\bB)$ via axioms $\EXIVa$ (item 1), $\EXIVb$ (item 2), $\EXIVc$ (item 3), 
        $\EXIVd$ (item 5), $\EXIVe$ (item 7), and $\EXIVf$ (item 8). In each case, the extension $f_i$ of $f$ defined by $a_i\mapsto b$
        is an embedding due to the moreover statement of Lemma~\ref{lem:XYZ}.
    \qedhere 
    \end{itemize}
\end{proof}

\subsection{Characterization and applications}

All the hard (technical) work of this section has been done above. Now, we build on it to obtain
our main results and applications.

\begin{theorem}\label{thm:extending-iso}
    Let $\bA$ and $\bB$ be two countable $\{E,B,S\}$-structures satisfying the universal axioms and the extension axioms and let $f \colon \bA' \to \bB'$ be an
    isomorphism between a finite 
    substructure $\bA'$ of $\bA$ and a finite substructure $\bB'$ of $\bB$. Then $f$ can be extended to an isomorphism between $\bA$ and $\bB$.
\end{theorem}
\begin{proof}
    Clearly, the isomorphism $f\colon \bA'\to \bB'$ defines an embedding of $\bA'$ into $\bB$. Thus, the isomorphism
    between $\bA$ and $\bB$ can be constructed via a back-and-forth argument using Theorem~\ref{thm:extending-emb},
    and the assumption that both structures $\bA$, $\bB$ satisfy the universal and the extension axioms. 
\end{proof}

Theorem~\ref{thm:extending-iso} implies that there is a unique countable $\{E,B,S\}$-structure
satisfying all universal and extension axioms, up to isomorphism, and moreover, this structure
is homogeneous. We now observe that $(\mathbb C_3,B,S)$ is such a structure. 

\begin{theorem}\label{thm:ax}
 A countable $\{E,B,S\}$-structure satisfies the universal and extension axioms
 if and only if it is isomorphic to $(\mathbb C_3, B,S)$.
\end{theorem}
\begin{proof}
    This follows from the fact that $(\mathbb C_3, B,S)$ satisfies these axioms (Observations~\ref{obs:ax}
    and~\ref{obs:ex-ax}) and Theorem~\ref{thm:extending-iso}.
\end{proof}

\begin{corollary}\label{cor:homogeneous}
    $(\mathbb C_3, B,S)$ is a homogeneous structure. 
\end{corollary}

Let $\Phi$ be the (finite!) set of first-order sentences over the signature of graphs which is obtained
from the universal and existential axioms by replacing each occurrence of
$B(x,y,z)$ and $S(x,y,z,w)$ by their existential definitions (Observations~\ref{obs:B-definitions}
and Observation~\ref{obs:S-definitions}, respectively). 

\begin{corollary}\label{cor:ax}
    A countable graph satisfies $\Phi$ if and only if it is isomorphic to ${\mathbb C}_3$; in particular, 
    ${\mathbb C}_3$ is $\omega$-categorical. 
\end{corollary}

\begin{corollary}\label{cor:mt}
    The structure ${\mathbb C}_3$ is a \emph{model-complete core}, i.e., for every finite $F \subseteq V({\mathbb C}_3)$ and every
    endomorphism $e$ of ${\mathbb C}_3$ there exists an automorphism $a$ of
    ${\mathbb C}_3$ such that $e|_F = a|_F$.\footnote{Equivalently, a structure is a model-complete core if every first-order formula is equivalent to an existential positive formula; see~\cite{Book}.}
\end{corollary} 
\begin{proof}
    Every structure $\bA$ with a homogeneous expansion whose relations are 
    existentially positively and universally negatively definable  in $\bA$ is a model-complete core~\cite[Theorem 4.5.1]{Book}.
    Corollary~\ref{cor:homogeneous} asserts that $({\mathbb C}_3, B, S)$ is a homogeneous expansion of $\mathbb C_3$,
    and Observation~\ref{obs:B-definitions} implies that $B$ is existentially and universally definable in $\mathbb C_3$.
    To see that it is existentially positive and universally negatively definable, first   notice that $\exists z (E(x,z)\land E(y,z))$ defines
    $\lnot E(x,y)$ in $\mathbb C_3$, while $\forall z (\lnot E(x,z)\lor
    \lnot E(y,z))$ defines $E(x,y)$. Second, in both definitions of $B$ and $S$ we implicitly use $x\neq y$ (which is a universal negative formula), and
    it can be existentially positively defined by $\exists z (E(z,x) \land \lnot E(z,y))$ (composing with the existential positive definition of $\lnot E$).
    Using this together with Observation~\ref{obs:S-definitions} we conclude that $S$ is 
    existentially positive and universally negative definable in $\mathbb C_3$.
\end{proof}

\begin{corollary}
    Every countable $\{K_3,~K_1+2K_2,~K_1+C_5,~C_6\}$-free graph embeds into ${\mathbb C}_3$.
\end{corollary}
\begin{proof}
    Follows from the fact that $\Age({\mathbb C}_3)$ equals the class of all finite $\{K_3,~K_1+2K_2,~K_1+C_5,~C_6\}$-free graphs (Theorem~\ref{thm:ageC3}) and the $\omega$-categoricity of ${\mathbb C}_3$ (Corollary~\ref{cor:ax}) 
    via a compactness argument (see, e.g.,~\cite[Lemma 4.1.7]{Book}). 
\end{proof}

\section{On the hardness of testing $\chi_c(G) < 3$}
\label{sec:complexity}
In this section we show that 
the problem of deciding
$\chi_c(G) < 3$ for a given finite graph $G$ is NP-hard, answering a question that has been asked in~\cite{guzmanAMC438}.
We present two different NP-hardness proofs, which are both interesting from a theoretic perspective, as we will explain below. 

The first proof is by a reduction from a known NP-hard \emph{promise constraint satisfaction problem}, namely $\Pcsp(C_5,K_3)$~\cite[Theorem 6.7]{BartoBKO21}. 
For graphs (and more generally, for structures) $G$ and $H$ such that $G \to H$ the computational problem
$\Pcsp(G,H)$ is defined as follows. For a given finite input graph $I$, output `Yes' if $I \to C_5$, and output `No' if $I \not\to H$. If $I$ is such that $I \not\to C_5$ and $I \to H$,
then the answer can be arbitrary. Note that $\Pcsp(H,H)$ is the same as $\Csp(H)$. 
An important technique to solve $\Pcsp(G,H)$ in polynomial time (sometimes called the \emph{sandwich technique}) is to identify a graph $S$ such that $G \to S \to H$ and $\Csp(S)$ can be solved in polynomial time. In this case, $\Pcsp(G,H)$ is in P by the trivial reduction, by simply returning the answer given by the algorithm for $\Csp(S)$. 
It has been conjectured  in~\cite{BrakensiekGuruswami19} that if $\Pcsp(G,H)$ is in P, then there is a (possibly infinite!) graph $S$ such that $\Csp(S)$ is in P; in other words, it has been conjectured that the sandwich technique is necessary and sufficient for polynomial-time tractability of finite-domain PCSPs. 
It is known that infinite domains are necessary here: 
there are examples of finite structures $\mathbb A \to \mathbb B$ such that every structure $\mathbb S$ such that $\mathbb A \to \mathbb S \to \mathbb B$ and $\Csp({\mathbb S}) \in P$ must be infinite (unless P=NP)~\cite{BartoBKO21}. 

Here, we use the sandwich technique in the other direction, to prove the NP-hardness of $\Csp(G)$ for an infinite graph $G$ 
(we are not aware of a previous application of the connection between PCSPs and infinite-domain CSPs in this direction in the literature). 

\begin{theorem}\label{thm:NP-hard}
    $\Csp({\mathbb C}_3)$ is NP-complete.
\end{theorem}
\begin{proof} It is straightforward to observe that $C_5$ has a homomorphism to ${\mathbb C}_3$, and that
${\mathbb C}_3$ has a homomorphism to $K_3$. 
The NP-hardness thus follows from the NP-hardness of $\Pcsp(C_5,K_3)$ proved in~\cite{BartoBKO21},
and containment in NP follows from the equivalence between the second and third item in
Theorem~\ref{thm:equivalent-extensions}.
\end{proof}

This hardness proof is interesting from the perspective of infinite-domain constraint satisfaction.
We have seen in Section~\ref{sec:expansion} that ${\mathbb C_3}$ can be expanded by first-order definable relations so
that the expanded structure is homogeneous.
Moreover, it is easy to see that a structure is finitely bounded and has a finite first-order expansion to a homogeneous structure, then 
the expansion is finitely bounded as well (this follows easily from the fact that homogeneous structures with a finite relational signature have quantifier elimination). 
Since the age of ${\mathbb C}_3$ is finitely bounded by 
the main results of the present paper (Theorem~\ref{thm:ageC3}), we can therefore conclude that 
$\Csp({\mathbb C}_3)$ falls into the scope of the \emph{tractability conjecture} from~\cite{BPP-projective-homomorphisms}, which states
that $\Csp({\mathbb C}_3)$ is in P unless it pp-constructs $K_3$, in which case it is NP-hard (the conjecture originally has been phrased
differently, but is equivalent to the form stated above by a result from~\cite{BKOPP}). 
We will present a graph-theoretic description of the notion of pp-constructability below (Definition~\ref{def:pp-construct}). 

For all known NP-hard CSPs that fall into the scope of the tractability conjecture, a pp-construction of $K_3$ is known (for the special case of finite structures, this is a consequence of the dichotomy theorem of Bulatov~\cite{BulatovFVConjecture} and of Zhuk~\cite{ZhukFVConjecture,Zhuk20}). 
Our NP-hardness proof based on the reduction from $\Pcsp(C_5,K_3)$ is different in that it does not give rise to a pp-construction of $K_3$. 

\begin{definition}\label{def:pp-construct}
    Let $G$ be a (not necessarily finite) graph. Then a graph $H$ has a \emph{primitive positive (pp) construction} in $G$ if there exist
    \begin{itemize}
        \item $d \in {\mathbb N}$ (the \emph{dimension} of the construction),
        \item a finite graph $K$, and
        \item $2d$ distinguished vertices $a_1,\dots,a_d,b_1,\dots,b_d \in V(K)$ (not necessarily distinct)   
    \end{itemize}
    such that $H$ is homomorphically equivalent to the graph with vertex set $V(G)^d$ where $u = (u_1,\dots,u_d)$ and $v = (v_1,\dots,v_d)$ are adjacent if and only if there exists a homomorphism from $K$ to $G$ that maps $a = (a_1,\dots,a_d)$ to $u$ and $b = (b_1,\dots,b_d)$ to $v$. 
\end{definition}


Primitive positive constructions are defined analogously for general relational structures, see~\cite{wonderland}.
In particular,  primitive positive constructability is transitive and if a graph (structure) $H$ has a pp-construction in 
a graph (structure) $G$, then there is a logspace reduction from $\Csp(H)$ to $\Csp(G)$ (see~\cite{wonderland}; also see~\cite{maximal-digraphs,smooth-digraphs} for structural results about graphs ordered by pp-constructability). Primitive positive constructibility has been generalised to PCSPs in~\cite{BartoBKO21}. 
Again, if  $(G,H)$ pp-constructs $(G',H')$, then 
there is a polynomial-time reduction from $\Pcsp(G',H')$ to $\Pcsp(G,H)$. 
In particular, if $(G,H)$ pp-constructs $(K_3,K_3)$, then
there is a polynomial-time reduction from $\Pcsp(G,H)$ to $\Pcsp(K_3,K_3) = \Csp(K_3)$.

We mention that the proof of the hardness of $\Pcsp(C_5,K_3)$ presented in~\cite{BartoBKO21}
(answering a question from~\cite{BrakensiekGuruswami18}) is more sophisticated than a pp-construction
of $\Csp(K_3)$, but rather builds on a strengthening of Lov\'asz's theorem on the chromatic number of Kneser graphs.  
However, in our particular case, we are able to come up with a pp-construction of $K_3$ in ${\mathbb C}_3$.

\begin{theorem}
    There is a pp-construction of $K_3$ in ${\mathbb C}_3$. 
\end{theorem}
\begin{proof}[Proof sketch]
    Since the structure ${\mathbb C}_3$ is a model-complete core (Corollary~\ref{cor:mt}), the expansion of 
    ${\mathbb C}_3$ by a constant $c \in V({\mathbb C}_3)$ is pp-constructible in ${\mathbb C}_3$ (see~\cite{wonderland}). 
    Recall that the ternary relation $B$ has an existential positive definition on $\mathbb C_3$ (see the proof of
    Corollary~\ref{cor:mt}). Moreover, this definition has no disjunctions, i.e., it is primitive positive. Notice that the neighbourhood $N(v)$ of any vertex $v$, together with the restriction of $B$ to
    $N(v)$, is isomorphic to $({\mathbb Q},\Bet)$,  where
    $\Bet$ is the usual betweenness relation, 
    $$ \Bet = \{(a,b,c) \mid (a<b<c) \vee (c <b<a) \}.$$
    It thus follows that $\mathbb C_3$ pp-constructs $(\mathbb Q, \Bet)$. 
    Finally, it is well-known that $K_3$ has a pp-construction in $({\mathbb Q},\Bet)$ (see~\cite{Book}). 
    Composing pp-constructions, we obtain a pp-construction of $K_3$ in ${\mathbb C}_3$. 
\end{proof}

We leave the following as an open problem. 

\begin{conjecture}\label{conj:K3ppconstruct}
    Suppose that $G$ and $H$ are finite graphs and $S$ is a countably infinite graph which is a reduct of a finitely bounded homogeneous structure such that 
    $G \to S \to H$. If $\Pcsp(G,H)$ is NP-hard, then $K_3$ has a pp-construction in $S$. 
\end{conjecture}

Note that if P is different from NP, then the tractability conjecture mentioned above implies Conjecture~\ref{conj:K3ppconstruct}. 
The connection between $\Pcsp(G,H)$ for finite graphs $G,H$ and infinite-domain constraint satisfaction remains interesting even outside the realm of reducts of finitely bounded homogeneous structures. 
It has been conjectured in~\cite[Conjecture 1.2]{BrakensiekGuruswami18} that $\Pcsp(C_{2k+1},K_n)$ is NP-hard for every
$k \geq 1$ and $n \geq 3$. This conjecture  implies a positive answer to the following question, which however might be
easier to prove.

\begin{question}\label{quest:hard-graphs}
    Let $G$ be a (not necessarily finite) $k$-colourable graph, for some $k \geq 3$, which is not bipartite. 
    Is $\Csp(G)$ NP-hard? 
\end{question}

For finite $G$, this question has a positive answer by the Hell-Ne\v{s}et\v{r}il theorem~\cite{HellNesetril}.
Also, this question has a positive answer for $3$-colourable $G$ by the NP-hardness of $\Pcsp(C_{2n+1},K_3)$ for
every integer $n \ge 2$~\cite[Theorem 1.3]{krokhinSIAM23}. 
We mention that for $k \geq 3$, there are uncountably many non-bipartite $k$-colourable graphs $G$ with pairwise distinct CSPs,
so the CSP of some of these graphs must be undecidable, because there are only countably many Turing machines. 
Question~\ref{quest:hard-graphs} is open even for graphs that are reducts of finitely bounded homogeneous structures.

Finally, we briefly discuss that the previous connection between PCSPs and infinite graph CSPs
extends to a broader class of problems in algorithmic graph theory. We say that a class of 
graph $\mathcal C$ has bounded chromatic number if there is some integer $k$ such that $\chi(G) \le k$
for every graph $G\in \mathcal C$. With similar arguments as in the proof of Theorem~\ref{thm:NP-hard},
it follows that if $\Pcsp(H,K_k)$ is NP-hard, then, for any class $\mathcal C$ with bounded chromatic
number $k$ and containing all graphs $H'$ with $H' \to H$, it is
NP-hard to test whether an input graph $G$ belongs to $\mathcal C$.

\begin{example}
    Since $\Pcsp(K_3, K_5)$ is  NP-hard  (see, e.g.,~\cite{BartoBKO21}), it follows that for every class
    of graphs $\mathcal C$ with bounded chromatic number $5$, and containing all $3$-colourable graphs,
    it is NP-hard to test whether $G\in \mathcal C$ for an input graph $G$.
\end{example}

Similarly, the conjectured hardness of  $\Pcsp(K_3,K_n)$ for every integer
$n\ge 3$~\cite[Conjecture 1.2]{BrakensiekGuruswami18}
implies a negative answer to the following question (unless P $=$ NP), which we believe is of interest in it own right
from a graph algorithmic point of view.

\begin{question}
    Is there a class of finite graphs $\mathcal C$ with bounded chromatic number and containing all $3$-colourable graphs 
    such that containment in $\mathcal C$ is decidable in polynomial-time? 
\end{question}


\section*{Acknowledgments}
Several discussions with Benjamin Moore and Colin Jahel during different stages of this project are gratefully acknowledged.

\bibliographystyle{abbrv}
\bibliography{global.bib}

\end{document}